\documentclass[11pt]{article}
\usepackage[latin9]{inputenc}
\usepackage{geometry}
\usepackage{color}
\usepackage{wrapfig}
\usepackage{units}
\usepackage{amsmath}
\usepackage{amsthm}
\usepackage{amssymb}
\usepackage{graphicx}
\usepackage[unicode=true,pdfusetitle,
 bookmarks=true,bookmarksnumbered=false,bookmarksopen=false,
 breaklinks=false,pdfborder={0 0 1},backref=page,colorlinks=true]
 {hyperref}

\makeatletter
\numberwithin{equation}{section}
\numberwithin{figure}{section}
\theoremstyle{plain}
\newtheorem{thm}{\protect\theoremname}
\theoremstyle{definition}
\newtheorem{defn}[thm]{\protect\definitionname}
\theoremstyle{plain}
\newtheorem{cor}[thm]{\protect\corollaryname}
\theoremstyle{remark}
\newtheorem{rem}[thm]{\protect\remarkname}
\theoremstyle{plain}
\newtheorem{fact}[thm]{\protect\factname}
\theoremstyle{plain}
\newtheorem{prop}[thm]{\protect\propositionname}
\theoremstyle{remark}
\newtheorem{claim}[thm]{\protect\claimname}
\theoremstyle{plain}
\newtheorem{lem}[thm]{\protect\lemmaname}

\usepackage{euscript}
\usepackage{xypic}
\usepackage{color}
\usepackage{fullpage}

\usepackage[all]{xy}
\usepackage{colortbl}
\usepackage{graphics}
\newcommand{\filleddiamond}{\raisebox{0.17\height}{\scalebox{0.7}[0.4]{$\vdots$}}}

\theoremstyle{plain}
\newtheorem{ques}[thm]{Question}
\theoremstyle{plain}
\newtheorem{claim}[thm]{\protect\claimname}

\numberwithin{thm}{section}

\newcommand{\pf}{\mathfrak{pf}} 
\newcommand{\herm}{{\cal HERM}}   

\newcommand\tmpcolor{blue}
\newcommand\M{{\color{\tmpcolor}\mathcal{M}}} 
\newcommand\vR{{\color{\tmpcolor}\v{R}}}  
\newcommand\vC{{\color{\tmpcolor}\v{C}}} 
\newcommand{\multi}{{\mathcal MultiMatch}}

\newtheorem*{prop*}{Proposition}
\renewcommand\v[1]{{\color{\tmpcolor}\dot{#1}}}

\makeatother

\providecommand{\claimname}{Claim}
\providecommand{\corollaryname}{Corollary}
\providecommand{\definitionname}{Definition}
\providecommand{\factname}{Fact}
\providecommand{\lemmaname}{Lemma}
\providecommand{\propositionname}{Proposition}
\providecommand{\remarkname}{Remark}
\providecommand{\theoremname}{Theorem}

\begin{document}
\global\long\def\P{{\cal P} }%

\title{Ramanujan Coverings of Graphs}
\author{Chris Hall\thanks{Author Hall was partially supported by Simons Foundation award 245619
and IAS NSF grant DMS-1128155.},~ Doron Puder\thanks{Author Puder was supported by the Rothschild fellowship and by the
National Science Foundation under agreement No. CCF-1412958.} ~~and~ William F. Sawin\thanks{Author Sawin was supported by the National Science Foundation Graduate
Research Fellowship under Grant No. DGE-1148900.}}
\maketitle
\begin{abstract}
Let $G$ be a finite connected graph, and let $\rho$ be the spectral
radius of its universal cover. For example, if $G$ is $k$-regular
then $\rho=2\sqrt{k-1}$. We show that for every $r$, there is an
$r$-covering (a.k.a.~an $r$-lift) of $G$ where all the new eigenvalues
are bounded from above by $\rho$. It follows that a bipartite Ramanujan
graph has a Ramanujan $r$-covering for every $r$. This generalizes
the $r=2$ case due to Marcus, Spielman and Srivastava \cite{MSS13}.

Every $r$-covering of $G$ corresponds to a labeling of the edges
of $G$ by elements of the symmetric group $S_{r}$. We generalize
this notion to labeling the edges by elements of various groups and
present a broader scenario where Ramanujan coverings are guaranteed
to exist. 

In particular, this shows the existence of richer families of bipartite
Ramanujan graphs than was known before. Inspired by \cite{MSS13},
a crucial component of our proof is the existence of interlacing families
of polynomials for complex reflection groups. The core argument of
this component is taken from \cite{MSS15}.

Another important ingredient of our proof is a new generalization
of the matching polynomial of a graph. We define the $r$-th matching
polynomial of $G$ to be the average matching polynomial of all $r$-coverings
of $G$. We show this polynomial shares many properties with the original
matching polynomial. For example, it is real rooted with all its roots
inside $\left[-\rho,\rho\right]$.
\end{abstract}
\tableofcontents{}

\section{Introduction\label{sec:Introduction}}

\subsection{Ramanujan coverings}

Throughout this paper, we assume that $G$ is a finite, connected,
undirected graph on $n$ vertices and that $A_{G}$\marginpar{$A_{G}$}
is its adjacency matrix. The eigenvalues of $A_{G}$ are real and
we denote them by 
\[
\lambda_{n}\le\ldots\le\lambda_{2}\le\lambda_{1}=\pf\left(G\right),
\]
where $\lambda_{1}=\pf\left(G\right)$\marginpar{$\pf\left(G\right)$}
is the Perron-Frobenius eigenvalue of $A_{G}$, referred to as the
trivial eigenvalue. For example, $\pf\left(G\right)=k$ for $G$ $k$-regular.
The smallest eigenvalue, $\lambda_{n}$, is at least $-\pf\left(G\right)$,
with equality if and only if $G$ is bipartite. Denote by $\lambda\left(G\right)$\marginpar{$\lambda\left(G\right)$}
the largest absolute value of a non-trivial eigenvalue, namely $\lambda\left(G\right)=\max\left(\lambda_{2},-\lambda_{n}\right)$.
It is well known that $\lambda\left(G\right)$ provides a good estimate
to different expansion properties of $G$: the smaller $\lambda\left(G\right)$
is, the better expanding $G$ is (see \cite{HLW06,Pud15}). 

However, $\lambda\left(G\right)$ cannot be arbitrarily small. Let
$\rho\left(G\right)$\marginpar{$\rho\left(G\right)$} be the spectral
radius of the universal covering tree of $G$. For instance, if $G$
is $k$-regular then $\rho\left(G\right)=2\sqrt{k-1}$. It is known
that $\lambda\left(G\right)$ cannot be much smaller than $\rho\left(G\right)$,
so graphs with $\lambda\left(G\right)\le\rho\left(G\right)$ are considered
optimal expanders (we elaborate in Section \ref{subsec:Expander-and-Ramanujan}
below). Following \cite{LPS88,HLW06} they are called \textbf{Ramanujan
graphs}, and the interval $\left[-\rho\left(G\right),\rho\left(G\right)\right]$
called \textbf{the Ramanujan interval}. In the bipartite case, $\lambda\left(G\right)=\left|\lambda_{n}\right|=\pf\left(G\right)$
is large, but $G$ can still expand well in many senses (see Section
\ref{subsec:Expander-and-Ramanujan}), and the optimal scenario is
when all other eigenvalues are within the Ramanujan interval, namely,
when $\lambda_{n-1},\lambda_{n-2},\ldots,\lambda_{2}\in\left[-\rho\left(G\right),\rho\left(G\right)\right]$.
We call a bipartite graph with this property a \textbf{bipartite-Ramanujan
graph.} 

Let $H$ be a topological $r$-sheeted covering\footnote{For a definition of a covering map between graphs see, e.g., \cite[Section 6.8]{godsil2001algebraic}
or the page ``Covering graph'' in Wikipedia. For an explicit construction
of finite coverings see Section \ref{subsec:Group-labeling-of}.} of $G$ ($r$-covering in short) with covering map $p\colon H\to G$.
Let $V\left(G\right)$ denote the set of vertices of $G$. If $f\colon V\left(G\right)\to\mathbb{R}$
is an eigenfunction of $G$, then the composition $f\circ p$ is an
eigenfunction of $H$ with the same eigenvalue. Thus, out of the $rn$
eigenvalues of $H$ (considered as a multiset), $n$ are induced from
$G$ and are referred to as \textbf{\emph{old eigenvalues}}. The other
$\left(r-1\right)n$ are called the \textbf{\emph{new eigenvalues}}\emph{
}of $H$.
\begin{defn}
Let $H$ be a finite covering of a finite graph $G$. We say that
$H$ is a \textbf{Ramanujan Covering} of $G$ if all the new eigenvalues
of $H$ are in $\left[-\rho\left(G\right),\rho\left(G\right)\right]$.
We say $H$ is \emph{a }\textbf{one-sided Ramanujan covering} if all
the new eigenvalues are bounded from above\footnote{We could also define a one-sided Ramanujan covering as having all
its eigenvalues bounded from below\emph{ by $-\rho\left(G\right)$.}
Every result stated in the paper about these coverings would still
hold for the lower-bound case, unless stated otherwise.} by $\rho\left(G\right)$.
\end{defn}

The existence of infinitely many $k$-regular Ramanujan graphs for
every $k\ge3$ is a long-standing open question. Bilu and Linial \cite{BL06}
suggest the following approach to solving this conjecture: start with
your favorite $k$-regular Ramanujan graph (e.g.~the complete graph
on $k+1$ vertices) and construct an infinite tower of Ramanujan $2$-coverings.
They conjecture that every (regular) graph has a Ramanujan $2$-covering.
This approach turned out to be very useful in the groundbreaking result
of Marcus, Spielman and Srivastava \cite{MSS13}, who proved that
every graph has a one-sided Ramanujan 2-covering. This translates,
as explained below, to that there are infinitely many $k$-regular
\emph{bipartite} Ramanujan graphs of every degree $k$. The question
remains open with respect to fully (i.e.~non-bipartite) Ramanujan
graphs.

In this paper, we generalize the result of \cite{MSS13} to coverings
of every degree:
\begin{thm}
\label{thm:Every-graph-has-one-sided-d-ram-cover}Every connected,
loopless graph has a one-sided Ramanujan $r$-covering for every $r$.
\end{thm}

In fact, this result holds also for graphs with loops, as long as
they are regular (Proposition \ref{prop:regular graphs with loops}),
so the only obstruction is irregular graphs with loops. We stress
that throughout this paper, all statements involving graphs hold not
only for simple graphs, but also for graphs with multiple edges. Unless
otherwise stated, the results also hold for graphs with loops.

A finite graph is bipartite if and only if its spectrum is symmetric
around zero. In addition, every covering of a bipartite graph is bipartite.
Thus, every one-sided Ramanujan covering of a bipartite graph is,
in fact, a (fully) Ramanujan covering. Therefore,
\begin{cor}
\label{cor:Every-bipartite-has-ram-cover}Every connected bipartite
graph has a Ramanujan $r$-covering for every $r$.
\end{cor}

In the special case where the base graph is $\vcenter{\xymatrix@1@R=2pt@C=25pt{ 
\bullet \ar@{-}@/^0.4pc/[r]_{\filleddiamond} \ar@{-}@/^.2pc/[r] \ar@{-}@/_0.41pc/[r] \ar@{-}[r]  & \bullet}}$
 (two vertices with $k$ edges connecting them), Theorem \ref{thm:Every-graph-has-one-sided-d-ram-cover}
(and Corollary \ref{cor:Every-bipartite-has-ram-cover}) were shown
in \cite{MSS15}, using a very different argument. In this regard,
our result generalizes the 2-coverings result from \cite{MSS13} as
well as the more recent result from \cite{MSS15}. \medskip{}

In our view, the main contributions of this paper are the following.
First, our results shed new light on the work of Marcus-Spielman-Srivastava
\cite{MSS13}: we show there is nothing special about $r=2$ ($2$-covering
of graphs), and that with the right framework, the ideas can be generalized
to any $r\ge2$. Second, our main result shows the existence of richer
families of bipartite-Ramanujan graphs than was known before (see
Corollary \ref{cor:simple-ramanujan-graphs}). Third, we introduce
a more general framework of group-based coverings of graphs, extend
Theorem \ref{thm:Every-graph-has-one-sided-d-ram-cover} to a more
general setting and point to the heart of the matter -- Properties
$\left(\P1\right)$ and $\left(\P2\right)$ defined below. Finally,
we introduce the $d$-matching-polynomial which has nice properties
and seems to be an interesting object for its own right.

\subsection{Group labeling of graphs and Ramanujan coverings\label{subsec:Group-labeling-of}}

As mentioned above, we suppose that $G$ is undirected, yet we regard
it as an oriented graph. More precisely, we choose an orientation
for each edge in $G$, and we write $E^{+}(G)$\marginpar{$E^{+}\left(G\right)$}
for the resulting set of oriented edges and $E^{-}(G)$\marginpar{$E^{-}\left(G\right)$}
for the edges with the opposite orientation. Finally, if $e$ is an
edge in $E^{\pm}(G)$, we write $-e$\marginpar{$-e$} for the corresponding
edge in $E^{\mp}(G)$ with the opposite orientation, and we identify
$E(G)$ with the disjoint union $E^{+}(G)\sqcup E^{-}(G)$. We let
$h\left(e\right)$ and $t\left(e\right)$\marginpar{$h\left(e\right),t\left(e\right)$}
denote the head vertex and tail vertex of $e\in E\left(G\right)$,
respectively. We say that $G$ is an \textit{oriented undirected graph}.

\begin{wrapfigure}{R}{0.4\columnwidth}%
\noindent \begin{centering}
\includegraphics[viewport=0bp 100bp 300bp 450bp,scale=0.4]{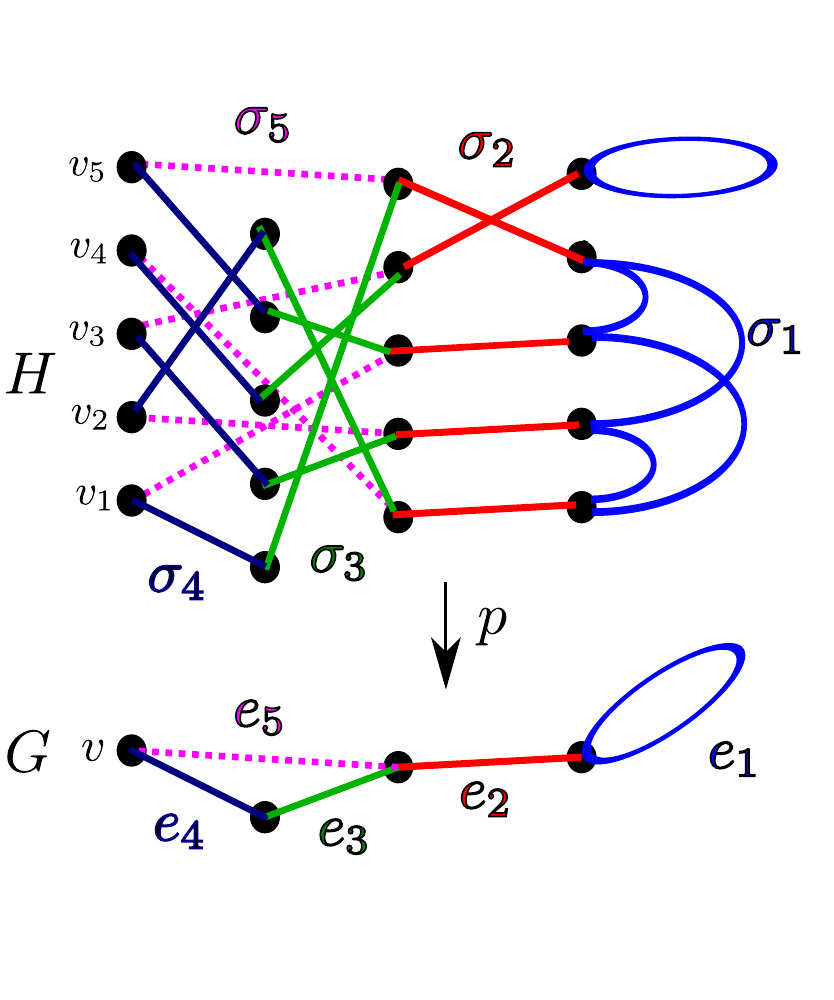}
\par\end{centering}
\caption{\label{fig:covering} A $5$-covering of a graph defined by permutations}
\end{wrapfigure}%
Throughout this paper, the family of $r$-coverings of the graph $G$
is defined via the following natural model, introduced in \cite{AL02}
and \cite{Fri03}. The vertex set of every $r$-covering $H$ is $\left\{ v_{i}\,\middle|\,v\in V\left(G\right),1\le i\le r\right\} $.
Its edges are defined via a function $\sigma\colon E\left(G\right)\to S_{r}$
satisfying $\sigma\left(-e\right)=\sigma\left(e\right)^{-1}$ (occasionally,
we denote $\sigma\left(e\right)$ by $\sigma_{e}$): for every $e\in E^{+}\left(G\right)$
we introduce in $H$ the $r$ edges connecting $t\left(e\right)_{i}$
to $h\left(e\right)_{\sigma_{e}\left(i\right)}$ for $1\le i\le r$.
See Figure \ref{fig:covering}.
\begin{defn}
\label{def:random-covering}For $G$ a finite graph and $r\in\mathbb{Z}_{\ge1}$,
denote by $\mathbf{{\cal C}_{r,G}}$\marginpar{${\cal C}_{r,G}$}
the probability space consisting of all $r$-coverings $\left\{ \sigma\colon E\left(G\right)\to S_{r}\,\middle|\,\sigma\left(-e\right)=\sigma\left(e\right)^{-1}\right\} $,
endowed with uniform distribution.
\end{defn}

Let $H\in{\cal C}_{r,G}$ correspond to $\sigma\colon E\left(G\right)\to S_{r}$
and let $f\colon V\left(H\right)\to\mathbb{C}$ be an eigenfunction
of $H$ with eigenvalue $\mu$. Define $\overline{f}\colon V\left(G\right)\to\mathbb{C}^{r}$
in terms of $f$: for every $v\in V\left(G\right)$, let $\overline{f}\left(v\right)$
be the transpose of the vector $\left(f\left(v_{1}\right),f\left(v_{2}\right),\ldots,f\left(v_{r}\right)\right)$.
Considering the permutations $\sigma_{e}$ as permutation matrices,
the collection of vectors $\left\{ \overline{f}\left(v\right)\right\} _{v\in V\left(G\right)}$
satisfies the following equation for every $v\in V\left(G\right)$:
\begin{equation}
\sum_{e:\,h\left(e\right)=v}\sigma_{e}\overline{f}\left(t\left(e\right)\right)=\mu\cdot\overline{f}\left(v\right)\label{eq:vector-eigenfunctions}
\end{equation}
(note that every loop at $v$ appears twice in the summation, once
in each orientation). Conversely, every function $f\colon V\left(G\right)\to\mathbb{C}^{r}$
satisfying \eqref{eq:vector-eigenfunctions} for some fixed $\mu$
and every $v\in V\left(G\right)$, is an eigenfunction of $H$ with
eigenvalue $\mu$.

This way of presenting coverings of $G$ and their spectra suggests
the following natural generalization: instead of picking the matrices
$\sigma_{e}$ from the group of permutation matrices, one can label
the edges of $G$ by matrices from any fixed subgroup of\footnote{We change the notation of dimension from $r$ to $d$ deliberately,
to avoid confusion: as explained below, the right point of view to
consider $r$-coverings is via a subgroup of $\mathrm{GL}_{r-1}\left(\mathbb{C}\right)$
rather than of $\mathrm{GL}_{r}\left(\mathbb{C}\right)$.} $\mathrm{GL}_{d}\left(\mathbb{C}\right)$. Since the same group $\Gamma$
may be embedded in different ways in $\mathrm{GL}_{d}\left(\mathbb{C}\right)$,
even for varying values of $d$, the right notion here is that of
group representations. Namely, a group $\Gamma$ together with a finite
dimensional representation $\pi$, which is simply a homomorphism
$\pi\colon\Gamma\to\mathrm{GL}_{d}\left(\mathbb{C}\right)$ (in this
case we say that $\pi$ is $d$-dimensional).
\begin{defn}
\label{def:gamma-pi-coverings}Let $\Gamma$ be a finite group. A
$\mathbf{\Gamma}$\textbf{-labeling} of the graph $G$ is a function
$\gamma\colon E\left(G\right)\to\Gamma$ satisfying $\gamma\left(-e\right)=\gamma\left(e\right)^{-1}$.
Denote by $\mathbf{{\cal C}_{\Gamma,G}}$\marginpar{${\cal C}_{\Gamma,G}$}
the probability space of all $\Gamma$-labelings of $G$ endowed with
uniform distribution.

Let $\pi\colon\Gamma\to\mathrm{GL}_{d}\left(\mathbb{C}\right)$ be
a representation of $\Gamma$. For any $\Gamma$-labeling $\gamma$
of $G$, we denote by $\mathbf{A_{\gamma,\pi}}$\marginpar{$A_{\gamma,\pi}$}
the $nd\times nd$ matrix obtained from $A_{G}$, the adjacency matrix
of $G$, as follows: for every $u,v\in V\left(G\right)$, replace
the $\left(u,v\right)$ entry in $A_{G}$ by the $d\times d$ block
$\sum_{e:u\to v}\pi\left(\gamma\left(e\right)\right)$ (the sum is
over all edges from $u$ to $v$, and is a zero $d\times d$ block
if there are no such edges). We say that $A_{\gamma,\pi}$ is a \emph{$\mathbf{\left(\Gamma,\pi\right)}$}\textbf{\emph{-covering}}
of the graph $G$. The $\mathbf{\pi}$\textbf{-spectrum} of the $\Gamma$-labeling
$\gamma$ is the spectrum of $A_{\gamma,\pi}$, namely, the multiset
of its eigenvalues, and is denoted $\mathrm{Spec}\left(A_{\gamma,\pi}\right)$\marginpar{$\mathrm{Spec}\left(A_{\gamma,\pi}\right)$}.
The fact that the $\pi$-spectrum is real is simple - see Claim \ref{claim: real spectrum}.

The $\left(\Gamma,\pi\right)$-covering $A_{\gamma,\pi}$ is said
to be \textbf{Ramanujan}\emph{ }if $\mathrm{Spec}\left(A_{\gamma,\pi}\right)\subseteq\left[-\rho\left(G\right),\rho\left(G\right)\right]$,
and \textbf{one-sided Ramanujan} if all the eigenvalues of $A_{\gamma,\pi}$
are at most $\rho\left(G\right)$. 
\end{defn}

Note that $\mu\in\mathrm{Spec}\left(A_{\gamma,\pi}\right)$ if and
only if there is some $0\ne f\colon V\left(G\right)\to\mathbb{C}^{d}$
satisfying
\[
\sum_{e:h\left(e\right)=v}\pi\left(\gamma\left(e\right)\right)f_{t\left(e\right)}=\mu\cdot f_{v}\,\,\,\,\,\,\,\,\,\forall v\in V\left(G\right),
\]
in accordance with \eqref{eq:vector-eigenfunctions}.

For example, if $G$ consists of a single vertex with several loops,
and if ${\cal R}$ denotes the \emph{regular} representation\footnote{\label{fn:regular rep}Namely, ${\cal R}$ is a $\left|\Gamma\right|$-dimensional
representation, and for every $g\in\Gamma$, the matrix ${\cal R}\left(g\right)$
is the permutation matrix describing the action of $g$ on the elements
of $\Gamma$ by right multiplication.} of $\Gamma$, then a $\left(\Gamma,{\cal R}\right)$-covering $A_{\gamma,{\cal R}}$
of $G$ is equivalent to the Cayley graph of $\Gamma$ with respect
to the set $\gamma\left(E\left(G\right)\right)$. The non-trivial
spectrum of this Cayley graph is given by the component ${\cal R}-\mathrm{triv}$
of ${\cal R}$ (see Section \ref{subsec:Group-Representations} for
some background). Hence, the Cayley graph is Ramanujan if and only
if the corresponding $\left(\Gamma,{\cal R}-\mathrm{triv}\right)$-covering
is Ramanujan .

As another example, the symmetric group $S_{r}$ has an $\left(r-1\right)$-dimensional
representation, called the \emph{standard }representation and denoted
$\mathrm{std}$ -- see Section \ref{subsec:Group-Representations}
for details. Every $r$-covering $H$ of $G$ corresponds to a $ $unique
$\left(S_{r},\mathrm{std}\right)$-covering, and, moreover, the new
spectrum of $H$ is precisely the spectrum of the corresponding $\left(S_{r},\mathrm{std}\right)$-covering
-- see Claim \ref{claim:equivalence-of-coverings-and-std-coverings}.
In particular, a Ramanujan $r$-covering corresponds to a Ramanujan
$\left(S_{r},\mathrm{std}\right)$-covering. When $r=2$, the standard
representation of $S_{2}\cong\mathbb{Z}/2\mathbb{Z}$ coincides with
the sign representation, and this correspondence between $2$-coverings
and $\left(S_{2},\mathrm{std}\right)$-coverings appears already in
\cite{BL06} and is used in \cite{MSS13}.

The following is, then, a natural generalization of the question concerning
ordinary Ramanujan coverings of graphs:

\begin{ques}\label{question:which-pairs-guarantee-Ramanujan-coverings}For
which pairs $\left(\Gamma,\pi\right)$ of a (finite) group $\Gamma$
with a representation $\pi\colon\Gamma\to\mathrm{GL}_{d}\left(\mathbb{C}\right)$
is it guaranteed that every connected graph $G$ has a (one-sided/fully)
Ramanujan $\left(\Gamma,\pi\right)$-covering?

\end{ques}

In this language Theorem \ref{thm:Every-graph-has-one-sided-d-ram-cover}
states that every connected graph has a one-sided Ramanujan $\left(S_{r},\mathrm{std}\right)$-covering
for every $r\ge2$. There are limitations to the possible positive
results one can hope for regarding Question \ref{question:which-pairs-guarantee-Ramanujan-coverings}
-- see Remark \ref{remark:abelian groups}. \medskip{}

Our proof of Theorem \ref{thm:Every-graph-has-one-sided-d-ram-cover}
exploits two group-theoretic properties of the pair $\left(S_{r},\mathrm{std}\right)$,
and this theorem can be generalized to any pair $\left(\Gamma,\pi\right)$
satisfying these two properties. The first property deals with the
exterior powers\footnote{See Section \ref{subsec:Group-Representations} for a definition of
exterior powers, irreducible representations and isomorphism of representations.} of $\pi$:
\begin{defn}
\label{def:P1}Let $\Gamma$ be a finite group and $\pi\colon\Gamma\to\mathrm{GL}_{d}\left(\mathbb{C}\right)$
a representation. We say that $\left(\Gamma,\pi\right)$ satisfies
$\mathbf{\left({\cal P}1\right)}$\marginpar{$\left({\cal P}1\right)$}
if all exterior powers $\bigwedge^{m}\pi,$ $0\le m\le d$, are irreducible
and non-isomorphic.
\end{defn}

The exterior power $\bigwedge^{0}\pi$ is always the trivial representation
mapping every $g\in\Gamma$ to $1\in\mathrm{GL}_{1}\left(\mathbb{C}\right)\cong\mathbb{C}^{*}$.
The next power, $\bigwedge^{1}\pi$, is simply $\pi$ itself. The
last power, $\bigwedge^{d}\pi$, is the one-dimensional representation
given by $\det\circ\pi\colon\Gamma\to\mathbb{C}^{*}$. Hence, if $\pi$
is one dimensional, $\left(\Gamma,\pi\right)$ satisfies $\left(\P1\right)$
if and only if $\pi$ is non-trivial. For example, the sign representation
of $\mathbb{Z}/2\mathbb{Z}$ used in \cite{MSS13} satisfies $\left(\P1\right)$.
If $\pi$ is 2-dimensional, $\left(\Gamma,\pi\right)$ satisfies $\left(\P1\right)$
if and only if $\pi$ is irreducible and $\pi\left(\Gamma\right)\nsubseteq\mathrm{SL}_{2}\left(\mathbb{C}\right)$.
We explain more in Section \ref{subsec:Group-Representations}.

Denote by $\mathbf{\phi_{\gamma,\pi}}$\marginpar{$\phi_{\gamma,\pi}$}
the characteristic polynomial of the $\left(\Gamma,\pi\right)$-covering
$A_{\gamma,\pi}$, namely 
\begin{equation}
\phi_{\gamma,\pi}\left(x\right)\overset{\mathrm{def}}{=}\det\left(xI-A_{\gamma,\pi}\right)=\prod_{\mu\in\mathrm{Spec\left(A_{\gamma,\pi}\right)}}\left(x-\mu\right).\label{eq:def-of-phi}
\end{equation}
Along this paper, the default distribution on $\Gamma$-labelings
of a graph $G$ is the one defined by ${\cal C}_{\Gamma,G}$ (see
Definition \ref{def:gamma-pi-coverings}). Hence, when $\Gamma$ and
$G$ are understood from the context, we use the notation $\mathbb{E}_{\gamma}\left[\phi_{\gamma,\pi}\left(x\right)\right]$
to denote the expected characteristic polynomial of a random $\left(\Gamma,\pi\right)$-covering,
the expectation being over the space ${\cal C}_{\Gamma,G}$ of $\Gamma$-labelings.

The following theorem describes the role of Property $\left({\cal P}1\right)$
in our proof:
\begin{thm}
\label{thm:P1}Let the graph $G$ be connected. For every pair $\left(\Gamma,\pi\right)$
satisfying $\left({\cal P}1\right)$ with $\dim\left(\pi\right)=d$,
the following holds:
\begin{equation}
\mathbb{E}_{\gamma}\left[\phi_{\gamma,\pi}\left(x\right)\right]=\mathbb{E}_{H\in{\cal C}_{d,G}}\M_{H}\left(x\right),\label{eq:thmP1}
\end{equation}
where $\M_{H}\left(x\right)$ is the matching polynomial\footnote{See \eqref{eq:matching poly}.}
of $H$. It particular, \textup{as long as $\left({\cal P}1\right)$
holds, $\mathbb{E}_{\gamma}\left[\phi_{\gamma,\pi}\left(x\right)\right]$
depends only on $d$ and not on $\left(\Gamma,\pi\right)$. }
\end{thm}

Namely, if $\left(\Gamma,\pi\right)$ satisfies $\left({\cal P}1\right)$,
then the expected characteristic polynomial of a random $\left(\Gamma,\pi\right)$-covering
is equal to the expected matching polynomial of a $d$-covering of
$G$. In particular, as we show below, $\mathrm{std}$ is a $d$-dimensional
representation of $S_{d+1}$ satisfying $\left({\cal P}1\right)$,
and so
\[
\frac{\mathbb{E}_{H\in{\cal C}_{d+1,G}}\left[\det\left[xI-A_{H}\right]\right]}{\det\left[xI-A_{G}\right]}=\mathbb{E}_{H\in{\cal C}_{d,G}}\M_{H}\left(x\right).
\]
This generalizes an old result from \cite{GG81} for the case $d=1$,
which is essential in \cite{MSS13}: the expected characteristic polynomial
of a $2$-covering of $G$ is equal to the characteristic polynomial
of $G$ times the matching polynomial of $G$. Together with Theorem
\ref{thm:M_d,G is Ramanujan} below, we get that whenever $\left(\Gamma,\pi\right)$
satisfies $\left({\cal P}1\right)$ and $G$ has no loops, the expected
characteristic polynomial $\mathbb{E}_{\gamma}\left[\phi_{\gamma,\pi}\left(x\right)\right]$
has only real-roots, all of which lie inside the Ramanujan interval
$\left[-\rho,\rho\right]$. We call the right hand side of \eqref{eq:thmP1}
the \textbf{$d$-matching polynomial }of $G$ -- see Definition \ref{def:d-matching-poly}.\medskip{}

To define the second property we need the notion of \emph{pseudo-reflections:
}a matrix $A\in\mathrm{GL}_{d}\left(\mathbb{C}\right)$ is called
a \textbf{\emph{pseudo-reflection}} if $A$ has finite order and $\mathrm{rank}\left(A-I\right)=1$.
Equivalently, $A$ is a pseudo-reflection if it is conjugate to a
diagonal matrix of the form 
\[
\left(\begin{array}{cccc}
\lambda\\
 & 1\\
 &  & \ddots\\
 &  &  & 1
\end{array}\right)
\]
with $\lambda\ne1$ some root of unity.
\begin{defn}
\label{def:P2}Let $\Gamma$ be a finite group and $\pi\colon\Gamma\to\mathrm{GL}_{d}\left(\mathbb{C}\right)$
a representation. We say that $\left(\Gamma,\pi\right)$ satisfies
$\mathbf{\left({\cal P}2\right)}$\marginpar{$\left({\cal P}2\right)$}
if $\pi\left(\Gamma\right)$ is a complex reflection group, namely,
if it is generated by pseudo-reflections.
\end{defn}

Complex reflection groups are a generalization of Coxeter groups.
The most well known example is the group of permutation matrices in
$\mathrm{GL}_{d}\left(\mathbb{C}\right)$: this group is generated
by transpositions which are genuine reflections. In the related case
of $\left(S_{r},\mathrm{std}\right)$, the image of every transposition
is a pseudo-reflection as well, hence $\left(S_{r},\mathrm{std}\right)$
satisfies $\left(\P2\right)$ -- see Section \ref{subsec:Group-Representations}
for details. The complete classification of pairs $\left(\Gamma,\pi\right)$
satisfying $\left(\P2\right)$ is well known -- see Section \ref{sec:applications}.

The following theorem manifests the role of $\left({\cal P}2\right)$
in our proof:
\begin{thm}
\label{thm:P2}Let $G$ be a finite, loopless graph. For every pair
$\left(\Gamma,\pi\right)$ satisfying $\left({\cal P}2\right)$, the
following holds:

\begin{itemize}
\item $\mathbb{E}_{\gamma}\left[\phi_{\gamma,\pi}\left(x\right)\right]$
is real rooted.
\item There exists a $\left(\Gamma,\pi\right)$-covering $A_{\gamma,\pi}$
with largest eigenvalue at most the largest root of $\mathbb{E}_{\gamma}\left[\phi_{\gamma,\pi}\left(x\right)\right]$.
\end{itemize}
\end{thm}

The proof of Theorem \ref{thm:P2} is based on the method of interlacing
polynomials. The core of the argument is inspired by \cite[Theorem 3.3]{MSS15}. 

\medskip{}

\noindent We can now state our generalization of Theorem \ref{thm:Every-graph-has-one-sided-d-ram-cover}:
\begin{thm}
\label{thm:gamma-pi-one-sided-rmnjn-covering}Let $\Gamma$ be a finite
group and $\pi\colon\Gamma\to\mathrm{GL}_{d}\left(\mathbb{C}\right)$
a representation such that $\left(\Gamma,\pi\right)$ satisfies $\left({\cal P}1\right)$
and $\left({\cal P}2\right)$. Then every connected, loopless graph
$G$ has a one-sided Ramanujan $\left(\Gamma,\pi\right)$\textup{-covering}.
\end{thm}

In Section \ref{sec:applications} we elaborate the combinatorial
consequences of Theorem \ref{thm:gamma-pi-one-sided-rmnjn-covering}.

\subsubsection*{Some Remarks}
\begin{rem}
\label{remark: compact}Some of the above definitions and results
apply also to compact groups $\Gamma$. For example, if $\Gamma$
is compact, we can let ${\cal C}_{\Gamma,G}$ be the probability space
of all $\Gamma$-labelings of $G$ endowed with Haar measure on $\Gamma^{E^{+}\left(G\right)}$.
If $\pi\colon\Gamma\to\mathrm{GL}_{d}\left(\mathbb{C}\right)$ is
unitary (see Section \ref{subsec:Group-Representations}), Property
$\left({\cal P}1\right)$ makes perfect sense, and Theorem \ref{thm:P1}
(and its proof) apply. Two interesting instances are specified in
Corollary \ref{cor:O(d) and U(d)}. There are also ways to generalize
$\left({\cal P}2\right)$ for compact groups in a way that the proof
of Theorem \ref{thm:P2} will apply.
\end{rem}

\begin{rem}
\label{remark:psl2}Suppose that $\pi$ is an $r$-dimensional permutation
representation of $\Gamma$, meaning that $\pi\left(\Gamma\right)$
is a group of permutation matrices in $\mathrm{GL}_{r}\left(\mathbb{C}\right)$
(see Section \ref{subsec:Permutation-Representations}). Suppose further
that $\left|\Gamma\right|$ is much smaller than $r!$. This means
that $\left(\Gamma,\pi\right)$-coverings of a graph $G$ correspond
to a small subset of all $r$-coverings of $G$. Every permutation
representation is composed of\footnote{See Section \ref{subsec:Group-Representations} for the meaning of
this.} the trivial representation and an $\left(r-1\right)$-dimensional
representation $\pi'$. A positive answer to Question \ref{question:which-pairs-guarantee-Ramanujan-coverings}
for $\pi'$ can yield a relatively fast algorithm to construct (bipartite-)
Ramanujan graphs.\\
A concrete compelling example is given by $\Gamma=\mathrm{PSL}_{2}\left(\mathbb{F}_{q}\right)$
and $\pi$ the action of $\Gamma$ by permutations on the projective
line $\mathbb{P}^{1}\left(\mathbb{F}_{q}\right)$. In this case, $\pi$
is $\left(q+1\right)$-dimensional, and $\pi'$ is irreducible of
dimension $q$. Here $\left|\Gamma\right|\approx q^{3}/2$. The pair
$\left(\mathrm{PSL}_{2}\left(\mathbb{F}_{q}\right),\pi'\right)$ satisfies
neither $\left(\P1\right)$ nor $\left(\P2\right)$, so the results
in this paper do not apply. However, there are only, roughly, $q^{3m}$
disjoint $\left(\mathrm{PSL}_{2}\left(\mathbb{F}_{q}\right),\pi'\right)$-coverings
of a graph $G$ with $m$ edges. So a positive answer to Question
\ref{question:which-pairs-guarantee-Ramanujan-coverings} in this
case means, for example, one can construct a Ramanujan $\left(q+1\right)$-covering
of the graph with $k$ edges connecting 2 vertices $\vcenter{\xymatrix@1@R=2pt@C=25pt{ 
\bullet \ar@{-}@/^0.4pc/[r]_{\filleddiamond} \ar@{-}@/^.2pc/[r] \ar@{-}@/_0.41pc/[r] \ar@{-}[r]  & \bullet}}$
, namely, a bipartite $k$-regular Ramanujan graph on $2\left(q+1\right)$
vertices, in time, roughly, $q^{3k}$. This example is especially
compelling because this subgroup of permutations is sparse and well-understood,
and also because the group $\mathrm{PSL}_{2}\left(q\right)$ has proven
useful before in constructing Ramanujan graphs: the explicit construction
of Ramanujan graphs in \cite{LPS88,Mar88} uses Cayley graphs of such
groups. 
\end{rem}

\begin{rem}
\label{remark:abelian groups}Returning to Question \ref{question:which-pairs-guarantee-Ramanujan-coverings},
we stress that not every pair $\left(\Gamma,\pi\right)$ guarantees
Ramanujan coverings. For example, let ${\cal R}$ denote the regular
representation of $\Gamma$ (see Footnote \ref{fn:regular rep}) and
assume that $\mathrm{rank}\left(\Gamma\right)>\mathrm{rank}\left(\pi_{1}\left(G\right)\right)$,
where $\mathrm{rank}\left(\Gamma\right)$ marks the minimal size of
a generating set of $\Gamma$. Then there is no surjective homomorphism
$\pi_{1}\left(G\right)\to\Gamma$ (see Claim \ref{claim: equivalences of labelings}),
so every $\left(\Gamma,{\cal R}\right)$-covering is necessarily disconnected,
and the new spectrum contains the trivial, Perron-Frobenius eigenvalue
of $G$.

Another counterexample to Question \ref{question:which-pairs-guarantee-Ramanujan-coverings}
is that of regular representations of abelian groups. In particular,
several authors asked about the existence of Ramanujan ``\textbf{shift
lifts}'' of graphs: whether every graph has an $r$-covering where
all edges are labeled by cyclic shift permutations, namely, by powers
of the permutation $\left(1\,2\,\ldots\,r\right)\in S_{r}$. This
is equivalent to $\left(C_{r},{\cal R}\right)$-coverings where $C_{r}$
is the cyclic group of size $r$. We claim there cannot be Ramanujan
$r$-coverings of this kind when $r$ is large. The reason is the
same argument showing that large abelian groups do not admit Ramanujan
Cayley graphs with a small number of generators: for example, the
balls in such a covering grow only polynomially, where in expander
graphs they grow exponentially fast. Concretely, simulations we conducted
show that the bouquet with one vertex and two loops has no Ramanujan
36-covering where the loops are labeled by cyclic permutations. It
also does not have any one-sided Ramanujan $73$-covering with cyclic
permutations. 
\end{rem}

\subsection{Overview of the proof\label{subsec:Overview-of-the}}

In this subsection we give a general outline of our proof. The definitions
and notations of this section are meant only to convey the spirit
of the arguments and are not used outside the current subsection.

The proof of Theorem \ref{thm:Every-graph-has-one-sided-d-ram-cover}
and its generalization, Theorem \ref{thm:gamma-pi-one-sided-rmnjn-covering},
follows the general proof strategy from \cite{MSS13}. A key point
in this strategy, is the following elementary yet very useful fact: 
\begin{fact}[{E.g.~proof of \cite[Lemma~4.2]{MSS13}}]
\label{fact:moving along straight lines} Assume that $f,g\in\mathbb{R}\left[x\right]$
are two polynomials of degree $n$ so that $\left(1-\lambda\right)f+\lambda g$
is real rooted for every $\lambda\in\left[0,1\right]$. Then, for
every $1\le i\le n$, the $i$-th root of $\left(1-\lambda\right)f+\lambda g$
moves monotonically when $\lambda$ moves from $0$ to $1$.
\end{fact}

Namely, if the roots of a polynomial $h$ are all real and denoted
$r_{n}\left(h\right)\le\ldots\le r_{2}\left(h\right)\le r_{1}\left(h\right)$,
then Fact \ref{fact:moving along straight lines} means that the function
$\lambda\mapsto r_{i}\left(\left(1-\lambda\right)f+\lambda g\right)$
is monotone (non-decreasing or non-increasing) for every $i$. We
give some more background and references in Section \ref{subsec:Interlacing-Families-of}. 

The starting point of the strategy of \cite{MSS13} is that instead
of considering ``discrete'' coverings of the graph $G$, one can
consider convex combinations of coverings, or more precisely, convex
combinations of \emph{characteristic polynomials} of coverings. Concretely,
let $\Delta_{r}\left(G\right)$\marginpar{$\Delta_{r}\left(G\right)$}
be the simplex of all probability distributions on $r$-coverings
of the graph $G$. The simplex $\Delta_{r}\left(G\right)$ has $\left|{\cal C}_{r,G}\right|=\left(r!\right)^{\left|E^{+}\left(G\right)\right|}$
vertices. To the vertex corresponding to the $r$-covering $H\in{\cal C}_{r,G}$
we associate the characteristic polynomial of the \emph{new} spectrum
of $H$, namely,
\[
\phi_{H}=\frac{\det\left(xI-A_{H}\right)}{\det\left(xI-A_{G}\right)}=\prod_{\mu\in\mathrm{New\,Spectrum\,of}\,H}\left(x-\mu\right)=\det\left(xI-A_{\sigma,\mathrm{std}}\right)
\]
where $\mathrm{std}$ is the standard $\left(r-1\right)$-dimensional
representation of $S_{r}$, and $\sigma$ is the $S_{r}$-labeling
of $G$ corresponding to $H$.

Every point $p\in\Delta_{r}\left(G\right)$ is associated with a polynomial
$\phi_{p}$, the corresponding weighted average of the $\left\{ \phi_{H}\right\} _{H\in{\cal C}_{r,G}}$:
if $p=\sum_{H\in{\cal C}_{r,G}}a_{H}\cdot H$, then $\phi_{p}=\sum_{H\in{\cal C}_{r,G}}a_{H}\cdot\phi_{H}$.
The proof now consists of two main parts: \\
$\left(i\right)$ Find a Ramanujan point $p_{Ram}\in\Delta_{r}\left(G\right)$,
namely a point whose corresponding polynomial $\phi_{p_{Ram}}$ is
real rooted with all its roots inside the Ramanujan interval $\left[-\rho,\rho\right]$.\\
$\left(ii\right)$ find a real-rooted region inside $\Delta_{r}\left(G\right)$
which contains $p_{Ram}$ and which allows one to use Fact \ref{fact:moving along straight lines}
and move along straight lines from $p_{Ram}$ to one of the vertices. 

We now explain each part in greater detail and explain how it is obtained
in the current paper:

\subsubsection*{Part $\left(i\right)$: a Ramanujan point inside $\Delta_{r}\left(G\right)$}

In the case $r=2$ studied in \cite{MSS13}, the center of $\Delta_{r}\left(G\right)$
is a Ramanujan point. This is the point corresponding to the uniform
distribution over all $2$-coverings of $G$. As mentioned above,
the polynomial at this point is the matching polynomial of $G$ \cite{GG81}.
The fact that this polynomial is Ramanujan, i.e.~is real-rooted with
all its roots in $\left[-\rho,\rho\right]$, is a classical fact due
to Heilmann and Lieb \cite{heilmann1972theory} -- see Theorem \ref{thm:1-matching-poly-is-ramanujan}
below. 

It turns out that for larger values of $r$, the center point of $\Delta_{r}\left(G\right)$
is still Ramanujan, yet this result is new, more involved, and relies
on Property $\left(\P1\right)$. Theorem \ref{thm:P1} states that
the polynomial associated with the center point is the $\left(r-1\right)$-matching
polynomial of $G$, defined as the average of the matching polynomials
of all $\left(r-1\right)$-coverings of $G$. In Section \ref{subsec:The--Matching-Polynomial}
we explain why every real root of this polynomial lies in $\left[-\rho,\rho\right]$.
That all its roots are real follows from the fact that the center
point of $\Delta_{r}\left(G\right)$ is contained in the real-rooted
region we find in Part $\left(ii\right)$ below. Note that although
the Ramanujan point we find happens to be the center of the simplex,
this is not required in general by the proof strategy.

\subsubsection*{Part $\left(ii\right)$: a real-rooted region inside $\Delta_{r}\left(G\right)$}

The main technical part of \cite{MSS13} consists of showing that
if $p\in\Delta_{2}\left(G\right)$ is a probability distribution on
$2$-coverings of $G$ in which the $S_{2}$-label of every $e\in E^{+}\left(G\right)$
is chosen independently, then the associated polynomial $\phi_{p}$
is real rooted. This gives a real-rooted region inside $\Delta_{2}\left(G\right)$.
For example, if the $S_{2}$-label of every $e\in E^{+}\left(G\right)$
is uniform among the two possibilities, we get the center point, which
provides yet another proof for that the matching polynomial of $G$
is real rooted. 

\begin{wrapfigure}{R}{0.4\columnwidth}%
\noindent \begin{centering}
\includegraphics[viewport=100bp 250bp 550bp 570bp,scale=0.5]{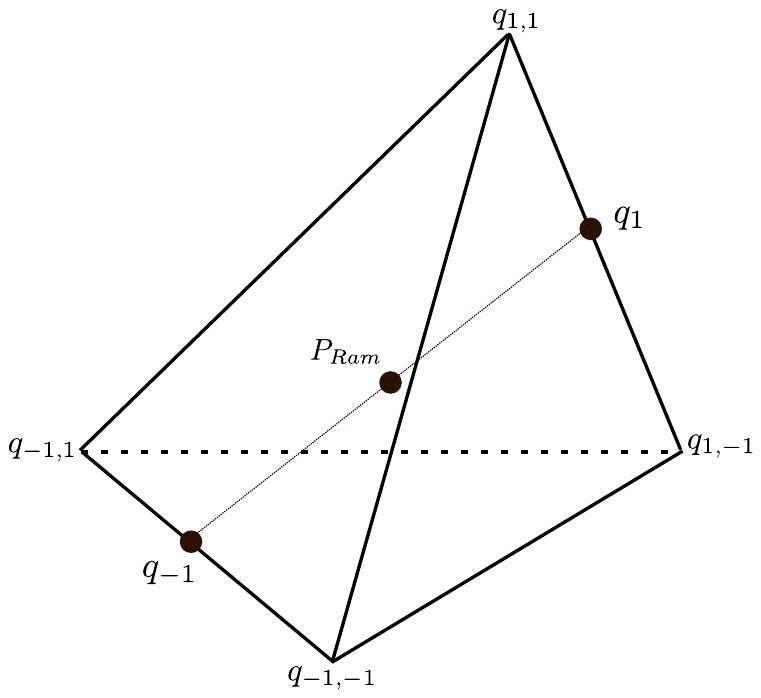}
\par\end{centering}
\caption{\label{fig:simplex} The simplex $\Delta_{2}\left(G\right)$ when
$\left|E^{+}\left(G\right)\right|=2$}
\end{wrapfigure}%
The crux of the matter is that among this family of distributions,
if one perturbs the distribution only on one particular edge $e\in E^{+}\left(G\right)$,
the corresponding point in $\Delta_{2}\left(G\right)$ moves along
a straight interval, which allows us to use Fact \ref{fact:moving along straight lines}. 

More concretely, order the edges in $E^{+}\left(G\right)$ by $e_{1},e_{2},\ldots$
in an arbitrary fashion. Start at the Ramanujan point $p_{Ram}$ --
the center point of $\Delta_{2}\left(G\right)$. Let $q_{1}$ denote
the point in $\Delta_{2}\left(G\right)$ where $e_{1}$ is labeled
by $\mathrm{id}\in S_{2}$ and all remaining edges are labeled uniformly
at random and independently. Let $q_{-1}$ be another point with the
same definition except that $e_{1}$ is labeled by $\left(1\,2\right)\in S_{2}$.
Now $p_{Ram}$ lies on the straight interval connecting $q_{1}$ and
$q_{-1}$. Note that every point on this interval corresponds to a
random 2-covering in which every edge is labeled independently from
the others, and the associated polynomial is, therefore, real rooted.
By Fact \ref{fact:moving along straight lines}, the largest root
of $\phi_{p_{Ram}}$ lies between the largest root of $\phi_{q_{1}}$
and the largest root of $\phi_{q_{-1}}$. Hence, one of the two points
$q_{1}$ or $q_{-1}$ has largest root at most the one of $p_{Ram}$,
and in particular at most $\rho$. Assume without lost of generality
that $q_{-1}$ has largest root at most $\rho$. Now repeat this process,
this time by perturbing the distribution on $e_{2}$. Let $q_{-1,1}\in\Delta_{2}\left(G\right)$
denote the point where $\sigma\left(e_{1}\right)=\left(1\,2\right)$,
$\sigma\left(e_{2}\right)=\mathrm{id}$ and all remaining edges are
distributed uniformly and independently, and let $q_{-1,-1}$ be defined
similarly, only with $\sigma\left(e_{2}\right)=\left(1\,2\right)$.
Since $q_{-1}$ lies on the straight line between $q_{-1,1}$ and
$q_{-1,-1}$, the largest root of one of these two latter points is
at most that of $q_{-1}$. If we go on and gradually choose a deterministic
value for every $e\in E^{+}\left(G\right)$ while not increasing the
largest root, we end up with a vertex of $\Delta_{2}\left(G\right)$
representing a one-side Ramanujan 2-covering of $G$. This is illustrated
in Figure \ref{fig:simplex}.

For larger values of $r$, the definition of the real-rooted region
is more subtle. Simple independence of the edges does not suffice\footnote{Consider, for instance, the cycle of length two $\vcenter{\xymatrix@1@R=2pt@C=25pt{ 
\bullet  \ar@{-}@/^.2pc/[r] \ar@{-}@/_0.2pc/[r]   & \bullet}}$
. Define a distribution $P$ on its $3$-coverings by labeling one
edge deterministically with the identity permutation and the other
edge with either the identity or a $3$-cycle $\left(1\,2\,3\right)$,
each with probability $\frac{1}{2}$. The average characteristic polynomial
of $A_{\sigma,\mathrm{std}}$ is then $\frac{\left(x^{2}-4\right)^{2}+\left(x^{2}-1\right)^{2}}{2}$
which is \emph{not} real rooted.}. Instead, we define the following real-rooted region in $\Delta_{r}\left(G\right)$.
This follows the ideas in Section 3 of \cite{MSS15}, and generalizes
the real-rooted region in $\Delta_{2}\left(G\right)$ from \cite{MSS13}.

\begin{prop*}  Let $p\in\Delta_{r}\left(G\right)$ be a probability
distribution of $r$-coverings of $G$ satisfying that for every $e\in E^{+}\left(G\right)$,
the random labeling of $e$
\begin{enumerate}
\item is independent of the labelings on other edges, and 
\item is equal to a product of independent random variables $X_{e,1},X_{e,2},\ldots,X_{e,\ell\left(e\right)}$,
where each $X_{e,i}$ has (at most) two values in its support: the
identity permutation and some transposition.
\end{enumerate}
Then $\phi_{p}$ is real rooted.\end{prop*} 

The statement of Proposition \ref{prop:avg-char-poly-of-random-coverings-is-real-rooted}
is slightly more general and applies to any pseudo-reflections and
not only transpositions. To illustrate this proposition, consider
the case $r=3$. Following \cite[Lemma 3.5]{MSS15}, define for every
$e\in E^{+}\left(G\right)$ three independent random variables $X_{e}$,
$Y_{e}$ and $Z_{e}$ taking values in $S_{3}$:

\begin{equation}
X_{e}=\begin{cases}
\mathrm{id} & \mathrm{with\,prob}\,\frac{1}{2}\\
\left(1\,2\right) & \mathrm{with\,prob}\,\frac{1}{2}
\end{cases}\,\,\,\,Y_{e}=\begin{cases}
\mathrm{id} & \mathrm{with\,prob}\,\frac{1}{3}\\
\left(1\,3\right) & \mathrm{with\,prob}\,\frac{2}{3}
\end{cases}\,\,\,\,Z_{e}=\begin{cases}
\mathrm{id} & \mathrm{with\,prob}\,\frac{1}{2}\\
\left(1\,2\right) & \mathrm{with\,prob}\,\frac{1}{2}
\end{cases}.\label{eq:XYZ}
\end{equation}
The random permutation $X_{e}\cdot Y_{e}\cdot Z_{e}$ has uniform
distribution in $S_{3}$. This shows that the center point $p\in\Delta_{3}\left(G\right)$
satisfies the assumptions in the Proposition and therefore $\phi_{p}$
is real rooted. Together with Theorem \ref{thm:P1} that implies that
all real roots of $\phi_{p}$ are in the Ramanujan interval $\left[-\rho,\rho\right]$,
we obtain that the center point is Ramanujan.

In Remark \ref{remark:middle point for Sym group} we explain how
the explicit construction of random variables in \eqref{eq:XYZ} can
be generalized to every $r$, showing that the center point of $\Delta_{r}\left(G\right)$
satisfies the assumptions in Proposition \ref{prop:avg-char-poly-of-random-coverings-is-real-rooted}
and is, therefore, real rooted. However, the crucial feature of the
family of $r$-coverings of $G$ is that $\mathrm{std}\left(S_{r}\right)\le\mathrm{GL}_{r-1}\left(\mathbb{C}\right)$
is a complex reflection group, i.e.~that $\left(S_{r},\mathrm{std}\right)$
satisfies $\left(\P2\right)$. In Section \ref{subsec:Proof-of-Theorem P2}
we give an alternative, non-constructive argument which shows that
for every pair $\left(\Gamma,\pi\right)$ satisfying $\left(\P2\right)$,
the center point of the corresponding simplex satisfies the assumptions
of Proposition \ref{prop:avg-char-poly-of-random-coverings-is-real-rooted}.

\medskip{}

Finally, to find a one-sided Ramanujan $r$-covering of $G$, we imitate
the process illustrated in Figure \ref{fig:simplex}, only at each
stage we perturb and then fix the value of one of the independent
random variables $X_{e,i}$. For example, in the case $r=3$ and the
variables constructed in \eqref{eq:XYZ}, this translates to the $3\left|E^{+}\left(G\right)\right|$-step
process illustrated in Figure \ref{fig:XYZ and interlacing tree}.
At each step, we determine the value of one of the $3\left|E^{+}\left(G\right)\right|$
variables so that the maximal root of the associated polynomial never
increases.

\begin{figure}
\includegraphics[viewport=0bp 400bp 300bp 650bp,scale=0.7]{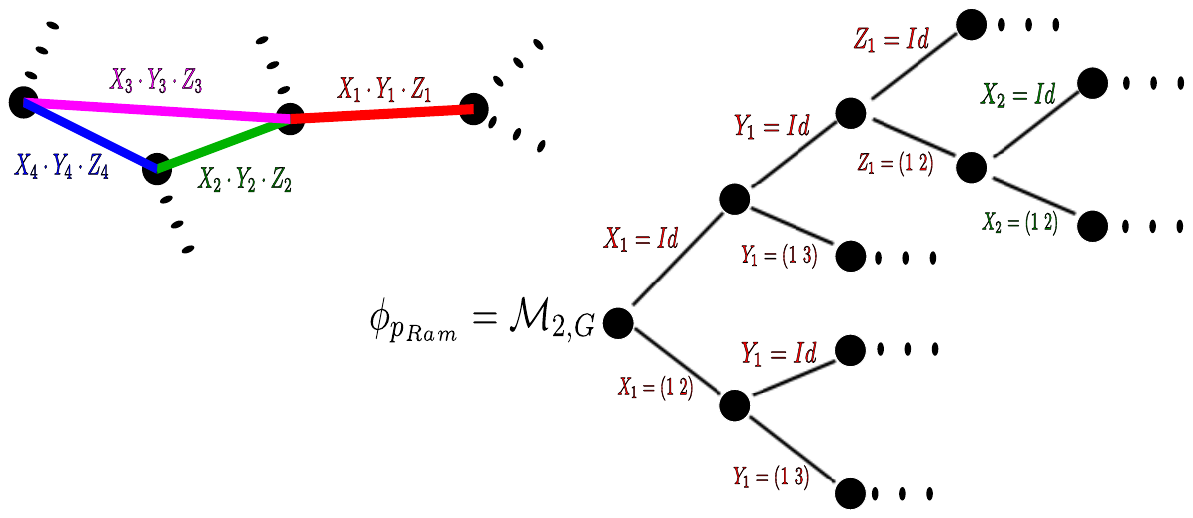}

\caption{On the left is a piece of the graph $G$ with three random variables
associated to every edge. On the right is a piece of the corresponding
binary tree of polynomials.}
\label{fig:XYZ and interlacing tree}
\end{figure}

\medskip{}

The rest of the paper is organized as follows. In Section \ref{sec:Background,-Preliminary-Claims}
we give more background details, prove some preliminary results and
reduce all results to proving Theorems \ref{thm:P1} and \ref{thm:P2}.
Section \ref{sec:P1} is dedicated to property $\left({\cal P}1\right)$
and the proof of Theorem \ref{thm:P1}, while in Section \ref{sec:P2}
we study Property $\left({\cal P}2\right)$ and prove Theorem \ref{thm:P2}.
In Section \ref{sec:applications}, we study groups satisfying the
two properties and present further combinatorial applications of Theorem
\ref{thm:gamma-pi-one-sided-rmnjn-covering}. We end in Section \ref{sec:Open-Questions}
with a list of open questions arising from the discussion in this
paper.

\section{Background and Preliminary Claims\label{sec:Background,-Preliminary-Claims}}

In this section we give more background material, prove some preliminary
claims, and reduce all the results from Section \ref{sec:Introduction}
to the proofs of Theorems \ref{thm:P1} and \ref{thm:P2}.

\subsection{Expander and Ramanujan Graphs\label{subsec:Expander-and-Ramanujan}}

As in Section \ref{sec:Introduction}, let $G$ be a finite connected
graph on $n$ vertices and $A_{G}$ its adjacency matrix. Recall that
$\pf\left(G\right)$ is the Perron-Frobenius eigenvalue of $A_{G}$,
that $\lambda_{n}\le\ldots\le\lambda_{2}\le\lambda_{1}=\pf\left(G\right)$
are its entire spectrum, and that $\lambda\left(G\right)=max\left(\lambda_{2},-\lambda_{n}\right)$.
The graph $G$ is considered to be well-expanding if it is ``highly''
connected. This can be measured by different combinatorial properties
of $G$, most commonly by its Cheeger constant, by the rate of convergence
of a random walk on $G$, or by how well the number of edges between
any two sets of vertices approximates the corresponding number in
a random graph (via the so-called Expander Mixing Lemma)\footnote{In this sense, Ramanujan graphs resemble random graphs. The converse
is also true in certain regimes of random graphs: see \cite{Pud15}
and the references therein.}. All these properties can be measured, at least approximately, by
the spectrum of $G$, and especially by $\lambda\left(G\right)$ and
the spectral gap $\pf\left(G\right)-\lambda\left(G\right)$: the smaller
$\lambda\left(G\right)$ and the bigger the spectral gap is, the better
expanding $G$ is\footnote{More precisely, the Cheeger inequality relates the Cheeger constant
of a graph with the value of $\lambda_{2}\left(G\right)$.}. (See \cite{HLW06} and \cite[Appendix B]{Pud15} and the references
therein.)

Yet, $\lambda\left(G\right)$ cannot be arbitrarily small. Let $T$
be the universal covering tree of $G$. We think of all the finite
graphs covered by $T$ as one family. For example, for any $k\ge2$,
all finite $k$-regular graphs constitute a single such family of
graphs: they are all covered by the $k$-regular tree. Let $\rho\left(T\right)$
be the spectral radius of $T$. This number is the spectral radius
of the adjacency operator $A_{T}$ acting on $\ell^{2}\left(V\left(T\right)\right)$
by 
\[
\left(A_{T}f\right)\left(v\right)=\sum_{u\sim v}f\left(u\right),
\]
and for the $k$-regular tree, this number is $2\sqrt{k-1}$. The
spectral radius of $T$ plays an important role in the theory of expansion
of the corresponding family of graphs:
\begin{thm}[{This version appears in \cite[Theorem~6]{cioabua2006eigenvalues},
a slightly weaker version appeared already in \cite[Theorem~2.11]{Gre95}}]
\label{thm:Greenberg} Let $T$ be a tree with finite quotients and
$\rho$ its spectral radius. For every $\varepsilon>0$, there exists
$c=c\left(T,\varepsilon\right)$, $0<c<1$, such that if $G$ is a
finite graph with $n$ vertices which is covered by $T$, then at
least $cn$ of its eigenvalues satisfy $\lambda_{i}\ge\rho-\varepsilon$.\\
In particular, $\lambda\left(G\right)\ge\rho-o_{n}\left(1\right)$
(with the $o_{n}\left(1\right)$ term depending only on $T$).
\end{thm}

The last statement of the theorem, restricted to regular graphs, is
due to Alon-Boppana \cite{Nil91}. Thus, graphs $G$ satisfying $\lambda\left(G\right)\le\rho\left(G\right)$
are considered to be optimal expanders. Following the terminology
of \cite{LPS88}, they are called Ramanujan graphs. 

The seminal works \cite{LPS88,Mar88,Mor94} provide an infinite family
of $k$-regular Ramanujan graphs whenever $k-1$ is a prime power.
Lubotzky \cite[Problem 10.7.3]{lubotzky1994discrete} asked whether
for every $k\ge3$ there are infinitely many $k$-regular Ramanujan
graphs\footnote{In fact, Lubotzky's original definition of Ramanujan graphs included
also \emph{bipartite}-Ramanujan graphs. Thus, \cite{MSS13} answers
this question positively.}. We stress that this only hints at a much stronger phenomena. For
example, it is known \cite{FRIEDMAN} that as $n\to\infty$, almost
all $k$-regular graphs are nearly Ramanujan, in the following sense:
for every $\varepsilon>0$, the non-trivial spectrum of a random $k$-regular
graph falls in $\left[-\rho-\varepsilon,\rho+\varepsilon\right]$
with probability tending to $1$. Moreover, it is conjectured that
among all $k$-regular graphs on $n$ vertices the proportion of Ramanujan
graphs tends to a constant in $\left(0,1\right)$ as $n\to\infty$
(e.g.~\cite{MNS08}).

Recall that we consider families of finite graphs defined by a common
universal covering tree. In the regular case, every family has at
least one Ramanujan graph (e.g.~the complete graph on $k+1$ vertices).
Other families may contain no Ramanujan graphs at all. For example,
the family of $\left(k,\ell\right)$-biregular graphs, all covered
by the $\left(k,\ell\right)$-biregular tree, consists entirely of
bipartite graphs, so none of them is Ramanujan in the strict sense.
There also exist families with no Ramanujan graphs, not even bipartite-Ramanujan
ones \cite{LN98}. In these cases there are certain ``bad'' eigenvalues
outside the Ramanujan interval appearing in every finite graph in
the family. Still, it makes sense to look for optimal expanders under
these constraints, namely, for graphs in the family where all other
eigenvalues lie in the Ramanujan interval. For example, bipartite-Ramanujan
graphs are optimal expanders in many combinatorial senses within the
family of bipartite graphs (e.g. \cite[Lemma 4.2]{golubev2014spectrum}).
The strategy of constructing Ramanujan coverings fits this general
goal: find any graph in the family which is optimal (has all its values
in the Ramanujan interval except for the bad ones) and construct Ramanujan
coverings to obtain more optimal graphs in the same family. Of course,
connected coverings of a graph $G$ are covered by the same tree as
$G$.

Marcus, Spielman and Srivastava have already shown that every graph
has a one-sided Ramanujan $2$-covering \cite{MSS13}. Thus, if a
family of graphs contains at least one Ramanujan graph (bipartite
or not), then it has infinitely many bipartite-Ramanujan graphs\footnote{Given a Ramanujan graph, its ``double cover'' --- the 2-covering
with all permutations being non-identity --- is bipartite-Ramanujan.}. More recently, they have showed that for any $k\ge3$, the graph
$\vcenter{\xymatrix@1@R=2pt@C=25pt{ 
\bullet \ar@{-}@/^0.4pc/[r]_{\filleddiamond} \ar@{-}@/^.2pc/[r] \ar@{-}@/_0.41pc/[r] \ar@{-}[r]  & \bullet}}$
(two vertices with $k$ edges connecting them) has a Ramanujan $r$-covering
for every $r$ \cite{MSS15}. It follows there are $k$-regular bipartite-Ramanujan
graphs, not necessarily simple, on $2r$ vertices for every $r$.
Our proof to the more general result, Theorem \ref{thm:Every-graph-has-one-sided-d-ram-cover},
is very different. It also yields the existence of a richer family
of bipartite-Ramanujan graphs than was known before. 
\begin{cor}
\label{cor:simple-ramanujan-graphs}Consider the family of all finite
graphs which are covered by some given common universal covering tree.
If the family contains a (simple) bipartite-Ramanujan graph on $n$
vertices, then it also contains (simple, respectively) bipartite-Ramanujan
graphs on $nr$ vertices for every $r\in\mathbb{Z}_{\ge1}$. \\
In particular, \emph{there is a simple} $k$-regular, bipartite-Ramanujan
graph on $2kr$ vertices for every $r$.\\
There is also a simple, $\left(k,\ell\right)$-biregular, bipartite-Ramanujan
graph on $\left(k+\ell\right)r$ vertices for every $r$.
\end{cor}

The last statement follows by constructing Ramanujan $r$-coverings
of the full $k$-regular bipartite graph on $2k$ vertices, or of
the full $\left(k,\ell\right)$-biregular bipartite graph on $k+\ell$
vertices, both of which are bipartite-Ramanujan.

As of now, we cannot extend all the results in this paper to graphs
with loops (and see Question \ref{ques:Dealing-with-loops}). However,
we can extend Theorem \ref{thm:Every-graph-has-one-sided-d-ram-cover}
to regular graphs with loops. We now give the short proof of this
extension, assuming Theorem \ref{thm:Every-graph-has-one-sided-d-ram-cover}:
\begin{prop}
\label{prop:regular graphs with loops}Let $G$ be a regular finite
graph, possibly with loops. Then $G$ has a one-sided Ramanujan $r$-covering
for every $r$.
\end{prop}

We remark that in this proposition the proof does not yield the analogous
result for coverings with new spectrum bounded from below by $-\rho\left(G\right)$.
\begin{proof}
Let $G$ be any finite connected graph with $n$ vertices and $m$
edges. Subdivide each of its edges by introducing a new vertex in
its middle, to obtain a new, bipartite graph $H$, with $n$ vertices
on one side and $m$ on the other. Clearly, there is a one-to-one
correspondence between (isomorphism types of) $r$-coverings of $G$
and (isomorphism types of) $r$-coverings of $H$. The rank of $A_{H}$
is at most $2n$, so $H$ has eigenvalue $0$ with multiplicity (at
least) $m-n$. As $H$ is bipartite, the remaining $2n$ eigenvalues
are symmetric around zero, and their squares are the eigenvalues of
$A_{G}+D_{G}$, where $A_{G}$ is the adjacency matrix of $G$ and
$D_{G}$ is diagonal with the degrees of the vertices.

Now assume G is $k$-regular, and let $\mu$ be an eigenvalue of $G$.
Then $\pm\sqrt{\mu+k}$ are eigenvalues of $H$ (and these are precisely
all the eigenvalues of $H$, aside to the $m-n$ zeros). By Corollary
\ref{cor:Every-bipartite-has-ram-cover}, $H$ has a Ramanujan $r$-covering
$\hat{H}_{r}$ for every $r$. Since the spectral radius of the $\left(k,2\right)$-biregular
tree is\footnote{In general, the spectral radius of the $\left(k,\ell\right)$-biregular
tree is $\sqrt{k-1}+\sqrt{\ell-1}$. } $\sqrt{k-1}+1$, every eigenvalue $\mu$ of the corresponding $r$-covering
$\hat{G}_{r}$ satisfies $\sqrt{\mu+k}\le\sqrt{k-1}+1$, i.e., $\mu\le2\sqrt{k-1}$. 
\end{proof}
The exact same argument can be used to extend also the statement of
Theorem \ref{thm:gamma-pi-one-sided-rmnjn-covering} to regular graphs
with loops: if $G$ is regular, possibly with loops, and $\left(\Gamma,\pi\right)$
satisfies $\left({\cal P}1\right)$ and $\left({\cal P}2\right)$,
then $G$ has a one-sided Ramanujan $\left(\Gamma,\pi\right)$-covering.

\subsection{The $d$-Matching Polynomial\label{subsec:The--Matching-Polynomial}}

An important ingredient in our proof of Theorem \ref{thm:Every-graph-has-one-sided-d-ram-cover}
is a new family of polynomials associated to a given graph. These
polynomials generalize the well-known matching polynomial of a graph
defined by Heilmann and Lieb \cite{heilmann1972theory}: let $m_{i}$
be the number of matchings in $G$ with $i$ edges, and set $m_{0}=1$.
The matching polynomial of $G$ is
\begin{equation}
\M_{G}\left(x\right)\overset{\mathrm{def}}{=}\sum_{i=0}^{\left\lfloor n/2\right\rfloor }\left(-1\right)^{i}m_{i}x^{n-2i}\in\mathbb{Z}\left[x\right].\label{eq:matching poly}
\end{equation}
The following is a crucial ingredient in the proof of the main result
of \cite{MSS13}:
\begin{thm}
\label{thm:1-matching-poly-is-ramanujan}\cite{heilmann1972theory}\footnote{Actually, \cite{heilmann1972theory} shows that $\M_{G}$ satisfies
this statement only when $G$ is regular. Apparently, the case of
irregular graphs was first noticed in \cite{MSS13}, even though some
of the original proofs of \cite{heilmann1972theory} work in the irregular
case as well.} The matching polynomial $\M_{G}$ of every finite connected graph
$G$ is real rooted with all its roots lying in the Ramanujan interval
$\left[-\rho\left(G\right),\rho\left(G\right)\right]$.
\end{thm}

\begin{defn}
\label{def:d-matching-poly}Let $d\in\mathbb{Z}_{\ge1}$. The \textbf{$\mathbf{d}$-matching
polynomial} of a finite graph $G$, denoted $\M_{d,G}$\marginpar{$\M_{d,G}$},
is the average of the matching polynomials of all $d$-coverings of
$G$ (in ${\cal C}_{d,G}$ -- see Definition \ref{def:random-covering}). 
\end{defn}

For example, if $G$ is $K_{4}$ minus an edge, then 
\[
\M_{3,G}\left(x\right)=x^{12}-15x^{10}+81x^{8}-189x^{6}+180x^{4}-\frac{178}{3}x^{2}+4.
\]
Of course, $\M_{1,G}=\M_{G}$ is the usual matching polynomial of
$G$ (a graph is the only 1-covering of itself). Note that these generalized
matching polynomials of $G$ are monic, but their other coefficients
need not be integer valued. However, they seem to share many of the
nice properties of the usual matching polynomial. For instance, 
\begin{cor}
\label{cor:d-matching-poly has no non-ramanujan real zeros}Every
real root of $\M_{d,G}$ lies inside the Ramanujan interval $\left[-\rho\left(G\right),\rho\left(G\right)\right]$.
\end{cor}

\begin{proof}
Every covering of $G$ belongs to the same family as $G$ (even when
the covering is not connected, each component is covered by the same
tree as $G$). Recall that $n$ denotes the number of vertices of
$G$. The ordinary matching polynomial of every $H\in{\cal C}_{d,G}$
is a degree-$nd$ monic polynomial. By Theorem \ref{thm:1-matching-poly-is-ramanujan},
it is strictly positive in the interval $\left(\rho\left(G\right),\infty\right)$,
and is either strictly positive or strictly negative in $\left(-\infty,-\rho\left(G\right)\right)$
depending only on the parity of $nd$. The corollary now follows by
the definition of $\M_{d,G}$ as the average of such polynomials.
\end{proof}
In fact, all roots of $\M_{d,G}$ are real\footnote{Except when $G$ has loops and then we do not know if this necessarily
holds.}. For this, we use the full strength of Theorems \ref{thm:P1} and
\ref{thm:P2}. We show (Fact \ref{fact:std_satisfies_P1_and_P2} below)
that for every $d$, there is a pair $\left(\Gamma,\pi\right)$ of
a group $\Gamma$ and a $d$-dimensional representation $\pi$ which
satisfies both $\left({\cal P}1\right)$ and $\left({\cal P}2\right)$.
From theorem \ref{thm:P1} we obtain that $\M_{d,G}\left(x\right)=\mathbb{E}_{\gamma}\left[\phi_{\gamma,\pi}\left(x\right)\right]$
(the expectation over ${\cal C}_{\Gamma,G}$), and from Theorem \ref{thm:P2}
we obtain that $\mathbb{E}_{\gamma}\left[\phi_{\gamma,\pi}\left(x\right)\right]$
is real-rooted. We wonder if there is a more direct proof of the real-rootedness
of $\M_{d,G}$ (see Question \ref{ques:The-d-matching-polynomial}).
This yields:
\begin{thm}
\label{thm:M_d,G is Ramanujan}Let $G$ be a finite, connected\footnote{Connectivity here is required only because of the way $\rho\left(G\right)$
was defined. The real-rootedness holds for any finite graph. In the
general case, the $ $$d$-matching polynomial is the product of the
$d$-matching polynomials of the different connected components, so
the statement remains true for non-connected graphs if $\rho\left(G\right)$
is defined as the maximum of $\rho\left(G_{i}\right)$ over the different
components $G_{i}$ of $G$. }, loopless graph. For every $ $$d\in\mathbb{Z}_{\ge1}$, the polynomial
$\M_{d,G}$ is real rooted with all its roots contained in the Ramanujan
interval $\left[-\rho\left(G\right),\rho\left(G\right)\right]$.
\end{thm}

In the proof of Theorem \ref{thm:P1}, which gives an alternative
definition for $\M_{d,G}$, we will use a precise formula for this
polynomial which we now develop. Every $H\in{\cal C}_{d,G}$, a $d$-covering
of $G$, has exactly $d$ edges covering any specific edge in $G$,
and, likewise, $d$ vertices covering every vertex of $G$. Thus,
one can think of $\M_{d,G}$ as a generating function of multi-matchings
in $G$: each edge in $G$ can be picked multiple times so that each
vertex is covered by at most $d$ edges. We think of such a multi-matching
as a function $m\colon E^{+}\left(G\right)\to\mathbb{Z}_{\ge0}$.
The weight associated to every multi-matching $m$ is equal to the
average number of ordinary matchings projecting to $m$ in a random
$d$-covering of $G$. Namely, the weight is the average number of
matchings in $H\in{\cal C}_{d,G}$ with exactly $m\left(e\right)$
edges projecting to $e$, for every $e\in E^{+}\left(G\right)$.

To write an explicit formula, we extend $m$ to all $E\left(G\right)$
by $m\left(-e\right)=m\left(e\right)$. We also denote by $e_{v,1},\ldots,e_{v,\deg\left(v\right)}$
the edges in $E\left(G\right)$ emanating from a vertex $v\in V\left(G\right)$
(in an arbitrary order, loops at $v$ appearing twice, of course),
and by $m\left(v\right)$ the number of edges incident with $v$ in
the multi-matching. Namely,
\[
m\left(v\right)=\sum_{i=1}^{\deg\left(v\right)}m\left(e_{v,i}\right).
\]
Finally, we denote by $\left|m\right|$\marginpar{$\left|m\right|$}
the total number of edges in $m$ (with multiplicity), so $\left|m\right|=\sum_{e\in E^{+}\left(G\right)}m\left(e\right)$.
\begin{defn}
\label{def: d-multi-matching}A \textbf{$\mathbf{d}$-multi-matching}
of a graph $G$ is a function $m\colon E\left(G\right)\to\mathbb{Z}_{\ge0}$
with $m\left(-e\right)=m\left(e\right)$ for every $e\in E\left(G\right)$
and $m\left(v\right)\le d$ for every $v\in V\left(G\right)$. We
denote the set of $d$-multi-matchings of $G$ by $\multi_{d}\left(G\right)$\marginpar{${\scriptscriptstyle \multi_{d}\left(G\right)}$}.
\end{defn}

\begin{prop}
\label{prop:formula-for-M_d,G}Let $m$ be a $d$-multi-matching of
$G$. Denote\footnote{We use the notation $\binom{b}{a_{1},a_{2},\ldots,a_{k}}$ to denote
the multinomial coefficient $\frac{b!}{a_{1}!\ldots a_{k}!\left(b-\sum a_{i}\right)!}$.}
\begin{equation}
W_{d}\left(m\right)=\frac{\prod_{v\in V\left(G\right)}\binom{d}{m\left(e_{v,1}\right),\ldots,m\left(e_{v,\deg\left(v\right)}\right)}}{\prod_{e\in E^{+}\left(G\right)}\binom{d}{m\left(e\right)}}.\label{eq:weight-of-multi-matching}
\end{equation}
Then,
\begin{equation}
\M_{d,G}\left(x\right)=\sum_{m\in\multi_{d}\left(G\right)}\left(-1\right)^{\left|m\right|}\cdot W_{d}\left(m\right)\cdot x^{nd-2\left|m\right|}.\label{eq:M_d,G formula}
\end{equation}
\end{prop}

\begin{proof}
Every matching of a $d$-covering $H\in{\cal C}_{d,G}$ projects to
a unique multi-matching $m$ of $G$ covering every vertex of $G$
at most $d$ times. Thus, it is enough to show that $W_{d}\left(m\right)$
is exactly the average number of ordinary matchings projecting to
$m$ in a random $H\in{\cal C}_{d,G}$. Every such matching in $H$
contains exactly $m\left(e\right)$ edges in the fiber above every
$e\in E\left(G\right)$. Assume we know, for each $e\in E\left(G\right)$,
which \emph{vertices} in $H$ are covered by the $m\left(e\right)$
edges above it. So there are $m\left(e\right)$ specific vertices
in the fiber above $h\left(e\right)$, and $m\left(e\right)$ specific
vertices in the fiber above $t\left(e\right)$. The probability that
a random permutation in $S_{d}$ matches specific $m\left(e\right)$
elements in $\left\{ 1,\ldots,d\right\} $ to specific $m\left(e\right)$
elements in $\left\{ 1,\ldots,d\right\} $ is 
\[
\frac{m\left(e\right)!\left(d-m\left(e\right)\right)!}{d!}=\binom{d}{m\left(e\right)}^{-1}.
\]
Thus, the denominator of $W_{d}\left(m\right)$ is equal to the probability
that a random $d$-covering has a matching which projects to $m$
and agrees with the particular choice of vertices. We are done as
the numerator is exactly the number of possible choices of vertices.
(Recall that since we deal with ordinary matchings in $H$, every
vertex is covered by at most one edge, so the set of vertices in the
fiber above $v\in V\left(G\right)$ which are matched by the pre-image
of $e_{v,i}$ is disjoint from those covered by the pre-image of $e_{v,j}$
whenever $i\ne j$.) Finally, we remark that the formula and proof
remain valid also for graphs with multiple edges or loops.
\end{proof}
The proof of Theorem \ref{thm:P1} in Section \ref{sec:P1} will consist
of showing that $\mathbb{E}_{\gamma}\left[\phi_{\gamma}\left(x\right)\right]$
is equal to the expression in \eqref{eq:M_d,G formula}.

To summarize, here is what this paper shows about the $d$-matching
polynomial $\M_{d,G}$ of the graph $G$:
\begin{itemize}
\item It can be defined by any of the following:

\begin{enumerate}
\item $\mathbb{E}_{H\in{\cal C}_{d,G}}\left[\M_{H}\right]$ --- the average
matching polynomial of a random $d$-covering of $G$
\item $\mathbb{E}_{H\in{\cal C}_{d+1,G}}\left[\frac{\det\left(xI-A_{H}\right)}{\det\left(xI-A_{G}\right)}\right]=\mathbb{E}_{H\in{\cal C}_{d+1,G}}\left[\prod_{\mu\in\mathrm{newSpec\left(H\right)}}\left(x-\mu\right)\right]$
- the average ``new part'' of the characteristic polynomial of a
random $\left(d+1\right)$-covering $H$ of $G$
\item $\mathbb{E}_{\gamma\in{\cal C}_{\Gamma,G}}\left[\phi_{\gamma,\pi}\right]$
- the average characteristic polynomial of a random $\left(\Gamma,\pi\right)$-covering
of $G$ whenever $\left(\Gamma,\pi\right)$ satisfies $\left({\cal P}1\right)$
and $\pi$ is $d$-dimensional
\item $\sum_{m\in\multi_{d}\left(G\right)}\left(-1\right)^{\left|m\right|}\cdot W_{d}\left(m\right)\cdot x^{nd-2\left|m\right|}$,
with $W_{d}\left(m\right)$ defined as in \eqref{eq:weight-of-multi-matching}.
\end{enumerate}
\item If $G$ has no loops, then $\M_{d,G}$ is real-rooted with all its
roots in the Ramanujan interval.
\end{itemize}

\subsection{Group Labelings of Graphs}

The content of the current subsection is not used in the rest of the
paper with the exception of Remark \ref{remark:abelian groups}. We
choose to give it here, albeit without proofs, for the sake of completeness. 

The model ${\cal C}_{r,G}$ we use for a random $r$-covering of a
graph $G$ is based on a uniformly random labeling $\gamma\colon E\left(G\right)\to S_{r}$.
This is generalized in Definition \ref{def:gamma-pi-coverings} to
${\cal C}_{\Gamma,G}$, a probability-space of random $\Gamma$-labelings
of the graph $G$. There are natural equivalent ways to obtain the
same distribution on (isomorphism) types of $r$-coverings or $\Gamma$-labelings. 

Two $r$-covering $H_{1}$ and $H_{2}$ of $G$ are isomorphic if
there is a graph isomorphism between them which respects the covering
maps. A similar equivalence relation can be given for $\Gamma$-labelings.
This is the equivalence relation generated, for example, by the equivalence
of the following two labelings of the edges incident to some vertex:

\includegraphics[viewport=0bp 270bp 600bp 510bp,scale=0.4]{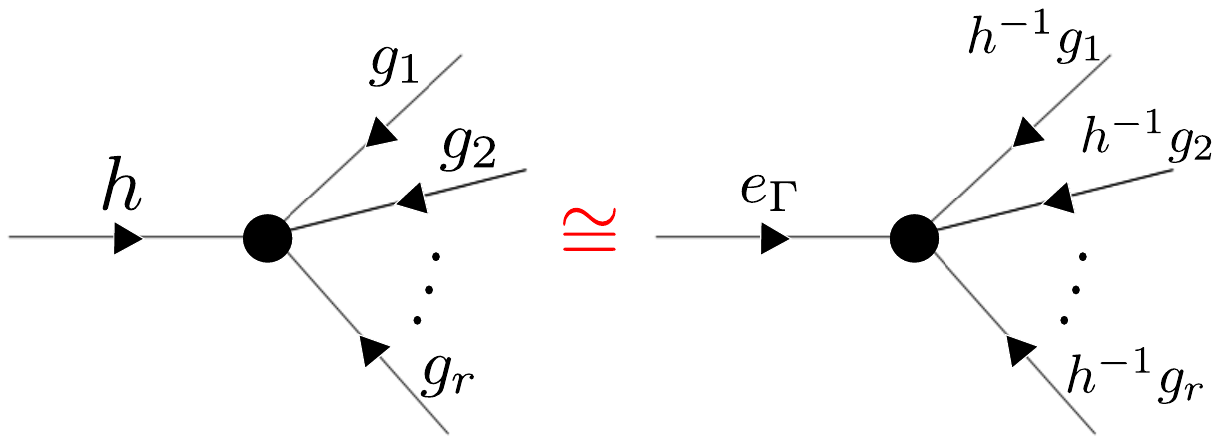}

\noindent (here $e_{\Gamma}$ is the identity element of $\Gamma$).
For example, if the $\Gamma$-labelings $\gamma_{1}$ and $\gamma_{2}$
of $G$ are isomorphic, then $\mathrm{Spec}\left(A_{\gamma_{1},\pi}\right)=\mathrm{Spec}\left(A_{\gamma_{2},\pi}\right)$
for any finite dimensional representation $\pi$ of $\Gamma$.
\begin{claim}
\label{claim: equivalences of labelings}Let $G$ be a finite connected
graph and $\Gamma$ a finite group. Let $T$ be a spanning tree of
$G$. The following three probability models yield the same distribution
on isomorphism types of $\Gamma$-labelings of $G$:

\begin{enumerate}
\item ${\cal C}_{\Gamma,G}$ --- uniform distribution on labelings $\gamma\colon E^{+}\left(G\right)\to\Gamma$
\item uniform distribution on homomorphisms $\pi_{1}\left(G\right)\to\Gamma$
\item an arbitrary fixed $\Gamma$-labeling of $E^{+}\left(T\right)$ (e.g.,
with the constant identity labeling) and a uniform distribution on
labelings of the remaining edges $E^{+}\left(G\right)\setminus E\left(T\right)$.
\end{enumerate}
\end{claim}

\subsection{Group Representations\label{subsec:Group-Representations}}

Let $\Gamma$ be a group. A (complex, finite-dimensional) \textbf{representation}\footnote{A standard reference for the subject of group representations is \cite{fulton1991representation}.}
of $\Gamma$ is any group homomorphism $\pi\colon\Gamma\to\mathrm{GL}_{d}\left(\mathbb{C}\right)$
for some $d\in\mathbb{Z}_{\ge1}$; if $\Gamma$ is a topological group,
we also demand $\pi$ to be continuous. We then say $\pi$ is a $d$-dimensional
representation. The representation is called \textbf{faithful} if
$\pi$ is injective. Two $d$-dimensional representations $\pi_{1}$
and $\pi_{2}$ are \textbf{isomorphic} if they are conjugate to each
other in the following sense: there is some $B\in\mathrm{GL}_{d}\left(\mathbb{C}\right)$
such that $\pi_{2}\left(g\right)=B^{-1}\pi_{1}\left(g\right)B$ for
every $g\in\Gamma$. The \textbf{trivial}\emph{ }representation is
the constant function $\mathrm{triv}\colon\Gamma\to\mathrm{GL}_{1}\left(\mathbb{C}\right)\cong\mathbb{C}^{*}$
mapping all elements to $1$. The direct sum of two representations
$\pi_{1}$ and $\pi_{2}$ of dimensions $d_{1}$ and $d_{2}$, respectively,
is a $\left(d_{1}+d_{2}\right)$-dimensional representation $\pi_{1}\oplus\pi_{2}\colon\Gamma\to\mathrm{GL}_{d_{1}+d_{2}}\left(\mathbb{C}\right)$
where $\left(\pi_{1}\oplus\pi_{2}\right)\left(g\right)$ is a block-diagonal
matrix, with a $d_{1}\times d_{1}$ block of $\pi_{1}\left(g\right)$
and a $d_{2}\times d_{2}$ block of $\pi_{2}\left(g\right)$. A representation
$\pi$ is called \textbf{irreducible} if is not isomorphic to the
direct sum of two representations\footnote{Equivalently, $\pi$ is irreducible if it has no non-trivial invariant
subspace, namely, no $\left\{ 0\right\} \ne W\lvertneqq\mathbb{C}^{d}$
with $\pi\left(g\right)\left(W\right)\le W$ for every $g\in\Gamma$.}. Otherwise, it is called \textbf{reducible}. 

Let $U\left(d\right)$ be the unitary group, that is, the subgroups
of matrices $A\in\mathrm{GL}_{d}\left(\mathbb{C}\right)$ whose inverse
is $A^{*}$, the conjugate-transpose of $A$. The representation $\pi$
is called \textbf{unitary} if its image in $\mathrm{GL}_{d}\left(\mathbb{C}\right)$
is conjugate to a subgroup of $U\left(d\right)$. In other words,
it is isomorphic to a representation $\Gamma\to U\left(d\right)$.
All representations of finite groups are unitary: e.g., conjugate
$\pi$ by $B=\left(\frac{1}{\left|\Gamma\right|}\sum_{g\in\Gamma}\pi\left(g\right)^{*}\pi\left(g\right)\right)^{1/2}$
to obtain a unitary image. 
\begin{claim}
\label{claim: real spectrum}Let $\pi$ be a unitary representation
of $\Gamma$ and $A_{\gamma,\pi}$ a $\left(\Gamma,\pi\right)$-covering
of some graph $G$. Then the spectrum of $A_{\gamma,\pi}$ is real.
\end{claim}

\begin{proof}
It is easy to see that $\mathrm{Spec\left(A_{\gamma,\pi}\right)=\mathrm{Spec}\left(A_{\gamma,\pi'}\right)}$
whenever $\pi$ and $\pi'$ are isomorphic. Thus, assume without loss
of generality that $\pi\left(\Gamma\right)\subseteq U\left(d\right)$.
Then, by definition, $A_{\gamma,\pi}$ is Hermitian, and the statement
follows.
\end{proof}
The $r$-dimensional representation $\pi$ of $S_{r}$ mapping every
$\sigma\in S_{r}$ to the corresponding permutation matrix is reducible:
the $1$-dimensional subspace of constant vectors $\left\langle \mathbf{1}\right\rangle \le\mathbb{C}^{r}$
is invariant under this representation. The action of this representation
on the orthogonal complement $\left\langle \mathbf{1}\right\rangle ^{\bot}$
is an $\left(r-1\right)$-dimensional irreducible representation of
$S_{r}$ called the \textbf{standard}\emph{ }representation and denoted
\textbf{$\mathrm{std}$}. The action on $\left\langle \mathbf{1}\right\rangle $
is isomorphic to the trivial representation. Thus, $\pi\cong\mathrm{std}\oplus\mathrm{triv}$.
\begin{claim}
\label{claim:equivalence-of-coverings-and-std-coverings}If $\gamma\colon E\left(G\right)\to S_{r}$
is an $S_{r}$-labeling of $G$, then the \emph{new} spectrum of the
$r$-covering of $G$ associated with $\gamma$ is equal to the spectrum
of $A_{\gamma,\mathrm{std}}$.\\
In particular, every (one-sided) Ramanujan $r$-covering of $G$ corresponds
to a unique (one-sided, respectively) Ramanujan $\left(S_{r},\mathrm{std}\right)$-covering
of $G$.
\end{claim}

\begin{proof}
For any $\Gamma$-labeling $\gamma$ of the graph $G$ and any two
representations $\pi_{1}$ and $\pi_{2}$, it is clear that $\mathrm{Spec}\left(A_{\gamma,\pi_{1}\oplus\pi_{2}}\right)$
is the disjoint union (as multisets) of $\mathrm{Spec}\left(A_{\gamma,\pi_{1}}\right)$
and $\mathrm{Spec}\left(A_{\gamma,\pi_{2}}\right)$. The claim follows
as $A_{\gamma,\mathrm{triv}}=A_{G}$ for any $\Gamma$-labeling $\gamma$.
\end{proof}
In this language, Theorem \ref{thm:Every-graph-has-one-sided-d-ram-cover}
says that every graph $G$ has a one-sided Ramanujan $\left(S_{r},\mathrm{std}\right)$-covering.
This theorem will follow from Theorem \ref{thm:gamma-pi-one-sided-rmnjn-covering}
if we show that the pair $\left(S_{r},\mathrm{std}\right)$ satisfies
both $\left({\cal P}1\right)$ and $\left({\cal P}2\right)$. Before
showing this, let us recall what exterior powers of representations
are.

Let $V=\mathbb{C}^{d}$. The $m$-th exterior power of $V$, $\bigwedge^{m}V$,
is the quotient of the tensor power $\bigotimes^{m}V$ by the subspace
spanned by $\left\{ v_{1}\otimes v_{2}\otimes\ldots\otimes v_{m}\,\middle|\,v_{i}=v_{j}\,\mathrm{for\,some}\,i\ne j\right\} $.
It is a $\binom{d}{m}$-dimensional vector space. The representative
of $v_{1}\otimes\ldots\otimes v_{m}$ is denoted $v_{1}\wedge\ldots\wedge v_{m}$
and we have $v_{\sigma\left(1\right)}\wedge v_{\sigma\left(2\right)}\wedge\ldots\wedge v_{\sigma\left(m\right)}=\mathrm{sgn}\left(\sigma\right)\cdot v_{1}\wedge v_{2}\wedge\ldots\wedge v_{m}$
for any permutation $\sigma\in S_{m}$.

Now let $\pi\colon\Gamma\to\mathrm{GL}_{d}\left(\mathbb{C}\right)$
be a $d$-dimensional representation. Its $m$-th exterior power,
denoted $\bigwedge^{m}\pi$, is a $\binom{d}{m}$-dimensional representation
depicting an action of $\Gamma$ on $\bigwedge^{m}V$. This action
is given by 
\[
g.\left(v_{1}\wedge\ldots\wedge v_{m}\right)\overset{\mathrm{def}}{=}\left(g.v_{1}\right)\wedge\ldots\wedge\left(g.v_{m}\right).
\]

\begin{fact}
\label{fact:std_satisfies_P1_and_P2}For every $r\in\mathbb{Z}_{\ge2}$,
the pair $\left(S_{r},\mathrm{std}\right)$ of the symmetric group
$S_{r}$ with its standard, $\left(r-1\right)$-dimensional representation
$\mathrm{std}$ satisfies both $\left({\cal P}1\right)$ and $\left({\cal P}2\right)$.
\end{fact}

\begin{proof}
That the exterior powers 
\[
\bigwedge\nolimits ^{\!0}\mathrm{std}=\mathrm{triv}\,\,,\,\,\bigwedge\nolimits ^{\!1}\mathrm{std}=\mathrm{std}\,\,,\,\,\bigwedge\nolimits ^{\!2}\mathrm{std\,\,},\,\,\ldots\,\,,\,\,\bigwedge\nolimits ^{\!r-1}\mathrm{std}=\mathrm{sign}
\]
of $\mathrm{std}$ are all irreducible and non-isomorphic to each
other is a classical fact. More concretely, $\bigwedge^{m}\mathrm{std}$
is the irreducible representation corresponding to the hook-shaped
Young diagram with $m+1$ rows $\left(r-m,1,1,\ldots,1\right)$, and
distinct Young diagrams correspond to distinct irreducible representations
(see Chapter 4 and, in particular, Exercise 4.6 in \cite{fulton1991representation}).
Hence $\left(S_{r},\mathrm{std}\right)$ satisfies $\left({\cal P}1\right)$.

The symmetric group $S_{r}$ is generated by transpositions (permutations
with $r-2$ fixed points and a single $2$-cycle). The image of a
transposition under $\pi\cong\mathrm{triv}\oplus\mathrm{std}$ is
a pseudo-reflection (with spectrum $\left\{ -1,1,1,\ldots,1\right\} $).
Because the spectrum of $\mathrm{triv}\left(\sigma\right)$ is $\left\{ 1\right\} $
for any $\sigma\in S_{r}$, we get that $\mathrm{Spec}\left(\mathrm{std}\left(\sigma\right)\right)=\left\{ -1,1,\ldots,1\right\} $
(with $r-2$ ones) whenever $\sigma$ is a transposition, namely,
$\mathrm{std}\left(\sigma\right)$ is a pseudo-reflection. Thus $\left(S_{r},\mathrm{std}\right)$
satisfies $\left({\cal P}2\right)$.
\end{proof}
Claim \ref{claim:equivalence-of-coverings-and-std-coverings} and
Fact \ref{fact:std_satisfies_P1_and_P2} show, then, why Theorem \ref{thm:Every-graph-has-one-sided-d-ram-cover}
is a special case of Theorem \ref{thm:gamma-pi-one-sided-rmnjn-covering}.
Fact \ref{fact:std_satisfies_P1_and_P2} also shows that for every
$d$ there is a pair $\left(\Gamma,\pi\right)$ satisfying $\left(\P1\right)$
and $\left(\P2\right)$ with $\dim\left(\pi\right)=d$. This, together
with Theorems \ref{thm:P1} and \ref{thm:P2}, yields that $\M_{d,G}$
is real-rooted for every loopless $G$ and every $d$. Adding Corollary
\ref{cor:d-matching-poly has no non-ramanujan real zeros} we obtain
Theorem \ref{thm:M_d,G is Ramanujan}. Since Theorem \ref{thm:gamma-pi-one-sided-rmnjn-covering}
follows from Theorems \ref{thm:P1}, \ref{thm:P2} and \ref{thm:M_d,G is Ramanujan},
it remains to prove Theorems \ref{thm:P1} and \ref{thm:P2}.

In Section \ref{sec:P1} below, we prove Theorem \ref{thm:P1} and
show that whenever the pair $\left(\Gamma,\pi\right)$ satisfies $\left({\cal P}1\right)$,
the polynomial $\mathbb{E}_{\gamma}\left[\phi_{\gamma,\pi}\right]$
is equal to $\M_{d,G}$. The crux of this proof is a calculation of
$\mathbb{E}_{\gamma}\left[\phi_{\gamma,\pi}\right]=\mathbb{E}_{\gamma}\left[\det\left(xI-A_{\gamma,\pi}\right)\right]$
by minors of the $d\times d$ blocks, noticing that the determinant
of an $m$-minor of $\pi\left(g\right)$ corresponds to an entry (matrix
coefficient) of $\left(\bigwedge^{m}\pi\right)\left(g\right)$, and
using the Peter-Weyl Theorem (Theorem \ref{thm:peter-weyl} below)
for matrix coefficients.

\subsection{Interlacing Polynomials\label{subsec:Interlacing-Families-of}}

A central theme of \cite{MSS13} as well as of the current paper is
showing that certain polynomials are real rooted. The main tool used
in the proof is that of polynomials with interlacing roots or polynomials
with common interlacing. The two elementary facts below, similar in
spirit, show that in certain situations interlacement is equivalent
to real-rootedness. Proofs can be found in \cite{fisk2006polynomials}.
Following \cite{MSS15}, we use these two facts in the proof of Theorem
\ref{thm:P2} in Section \ref{sec:P2}. 
\begin{defn}
\label{def:interlacing}Let $f,g\in\mathbb{R}\left[x\right]$ be real
rooted, $n=\deg\left(f\right)$ and $\alpha_{n}\le\ldots\le\alpha_{1}$
the roots of $f$.

\begin{enumerate}
\item \label{enu:deg(f)=00003Dn deg(g)=00003Dn-1}We say that $f$ and $g$
\textbf{interlace} if $\deg\left(g\right)=n-1$ and the roots $\beta_{n-1}\le\ldots\le\beta_{1}$
of $g$ satisfy 
\[
\alpha_{n}\le\beta_{n-1}\le\alpha_{n-1}\le\ldots\le\beta_{2}\le\alpha_{2}\le\beta_{1}\le\alpha_{1}.
\]
\item \label{enu:common interlacing}We say that $f$ and $g$\textbf{ have
common interlacing} if $\deg\left(g\right)=n$, its leading coefficient
has the same sign as that of $f$, and its roots $\beta_{n}\le\ldots\le\beta_{1}$
satisfying
\[
\left\{ \alpha_{n},\beta_{n}\right\} \le\left\{ \alpha_{n-1},\beta_{n-1}\right\} \le\ldots\le\left\{ \alpha_{2},\beta_{2}\right\} \le\left\{ \alpha_{1},\beta_{1}\right\} 
\]
(i.e., $\alpha_{i+1}\le\beta_{i}$ and $\beta_{i+1}\le\alpha_{i}$
for every $i$).
\end{enumerate}
\end{defn}

The second definition can be extended to any set of polynomials: the
$i$-th root of any of them is bigger than (or equal to) the $\left(i+1\right)$-st
root of any other. 
\begin{fact}
\label{fact:interlace}Let $f,g\in\mathbb{R}\left[x\right]$ with
$\deg\left(f\right)=n$ and $\deg\left(g\right)=n-1$. The polynomials
$f$ and $g$ interlace if and only if $f+\alpha g$ is real rooted
for every $\alpha\in\mathbb{R}$. \\
Moreover, in this case, the roots change monotonically as $\alpha$
grows.
\end{fact}

More precisely, give a real-rooted $h$, let $r_{\deg\left(h\right)}\left(h\right)\le\ldots\le r_{1}\left(h\right)$
be the roots of $h$, and write $r_{0}\left(h\right)=+\infty$ and
$r_{\deg\left(h\right)+1}\left(h\right)=-\infty$. With this notation,
if the leading coefficients of $f$ and $g$ have the same sign, then
$\alpha\mapsto r_{i}\left(f+\alpha g\right)$ is monotonically decreasing
for every $1\le i\le n$, with $r_{i}\left(f+\alpha g\right)$ decreasing
from $r_{i-1}\left(g\right)$ to $r_{i}\left(g\right)$. If the leading
coefficients have opposite signs, the roots are monotonically increasing.
\begin{fact}
\label{fact:common interlacing}Let $f_{1},\ldots,f_{m}\in\mathbb{R}\left[x\right]$
have the same degree $n$. These polynomials have a common interlacing
if and only if\footnote{\label{fn:interlacing iff probability measure-1}A more general claim,
just as easy, states that the set of polynomials $\left\{ f_{\alpha}\right\} _{\alpha\in A}\subseteq\mathbb{R}\left[x\right]$
is interlacing if and only if for any probability measure on $A$,
the expected polynomial is real rooted.} the average $\lambda_{1}f_{1}+\ldots+\lambda_{m}f_{m}$ is real rooted
for every $\lambda_{1},\ldots,\lambda_{m}$ with $\lambda_{i}\ge0$
and $\sum\lambda_{i}$=1.\\
Moreover, in this case, $r_{i}\left(\lambda_{1}f_{1}+\ldots+\lambda_{m}f_{m}\right)$
lies in the convex hull of $r_{i}\left(f_{1}\right),\ldots,r_{i}\left(f_{m}\right)$
for every $1\le i\le n$.
\end{fact}

The last property is a generalization of Fact \ref{fact:moving along straight lines}.
The simple argument of the proof appears, for example, in \cite[Lemma 4.2]{MSS13}.

\medskip{}

\section{Property $\left({\cal P}1\right)$ and the Proof of Theorem \ref{thm:P1}\label{sec:P1}}

Recall that $G$ is an undirected oriented graph with $n$ vertices.
In this section we assume the pair $\left(\Gamma,\pi\right)$ satisfies
$\left({\cal P}1\right)$, namely that $\Gamma$ is finite and $\pi$
is $d$-dimensional such that its exterior powers $\bigwedge^{0}\pi,\ldots,\bigwedge^{d}\pi$
are irreducible and non-isomorphic. We need to show that $\mathbb{E}_{\gamma\in{\cal C}_{\Gamma,G}}\left[\phi_{\gamma,\pi}\right]=\M_{d,G}$.
We stress the proof is valid also in the more general case of $\Gamma$
being compact and $\pi$ unitary -- see Remark \ref{remark: compact}
and Corollary \ref{cor:O(d) and U(d)}.

For every $\Gamma$-labeling $\gamma$ of $G$, we represent the matrix
$A_{\gamma,\pi}\in M_{nd}\left(\mathbb{C}\right)$ as a sum of $\left|E\left(G\right)\right|$
matrices as follows. For every $e\in E\left(G\right)$, let $\mathbf{A_{\gamma,\pi}\left(e\right)}\in\mathrm{M}_{nd}\left(\mathbb{C}\right)$\marginpar{$A_{\gamma,\pi}\left(e\right)$}\label{notation: A_gamma,pi}
be the $nd\times nd$ matrix composed of $n^{2}$ blocks of size $d\times d$
each. All blocks are zero blocks except for the one corresponding
to $e$, the block $\left(t\left(e\right),h\left(e\right)\right)$,
in which we put $\pi\left(\gamma\left(e\right)\right)$. Clearly,
\begin{equation}
A_{\gamma,\pi}=\sum_{e\in E\left(G\right)}A_{\gamma,\pi}\left(e\right).\label{eq:A_gamma_pi as sigma over oriented edges}
\end{equation}
In order to analyze the expected characteristic polynomial of this
sum of matrices, we begin with a technical lemma, giving the determinant
of a sum of matrices as a formula in terms of the determinants of
their minors. This lemma is used in Section \ref{subsec:proof-of-thm P1}
where we complete the proof of Theorem \ref{thm:P1}.

\subsection{\label{subsec:det of sum}Determinant of Sum of Matrices}

Let $A_{1},\ldots,A_{q}\in M_{d}\left(\mathbb{C}\right)$ be $d\times d$
matrices. The determinant $\left|A_{1}+\ldots+A_{q}\right|$ can be
thought of as a double sum. First, sum $\mathrm{sgn}\left(\sigma\right)\prod_{i=1}^{d}\left(A_{1}+\ldots+A_{q}\right)_{i,\sigma\left(i\right)}$
over all permutations $\sigma\in S_{d}$. Then, for each term and
each $i\in\left[d\right]=\left\{ 1,\ldots,d\right\} $, choose $s_{\sigma}\left(i\right)\in\left[q\right]$
which marks which of the $q$ summands is taken in the entry $\left(i,\sigma\left(i\right)\right)$.
Namely,
\begin{equation}
\left|A_{1}+\ldots+A_{q}\right|=\sum_{\sigma\in S_{d}}\mathrm{sgn}\left(\sigma\right)\sum_{s_{\sigma}\colon\left[d\right]\to\left[q\right]}\prod_{i=1}^{d}\left[A_{s_{\sigma}\left(i\right)}\right]_{i,\sigma\left(i\right)}.\label{eq:double-sum-for-det-of-sum}
\end{equation}
The idea of Lemma \ref{lem:det-of-sum} below is to group the terms
in this double sum in a different fashion: first, for every $j\in\left[q\right]$,
choose from which rows $R_{j}\subseteq\left[d\right]$ and from which
columns $C_{j}\subseteq\left[d\right]$ the entry is taken from $A_{j}$.
Then, vary over all permutations $\sigma$ that respect these constraints,
namely the permutations for which $\sigma\left(R_{j}\right)=C_{j}$.
With this in mind, we define:\marginpar{$T\left(q,d\right)$}

\[
T\left(q,d\right)=\left\{ \left(\vR,\vC\right)\,\middle|\,\begin{gathered}\vR=\left(R_{1},\ldots,R_{q}\right),\,\vC=\left(C_{1},\ldots,C_{q}\right)\,\,\mathrm{are\,partitions\,of}\,\left[d\right]\,\mathrm{into}\,q\,\mathrm{parts}\\
\mathrm{such\,that}\,\left|R_{\ell}\right|=\left|C_{\ell}\right|\,\,\mathrm{for\,all}\,\,1\le\ell\le q
\end{gathered}
\right\} ,
\]
and the corresponding permutations:\marginpar{$\mathrm{Sym}\left(\vR,\vC\right)$}
\[
\mathrm{Sym}\left(\vR,\vC\right)\overset{\mathrm{def}}{=}\left\{ \sigma\in S_{d}\,\middle|\,\sigma\left(R_{\ell}\right)=C_{\ell}\,\mathrm{for\,all}\,\ell\right\} .
\]
Finally, for each such pair of partitions $(\vR,\vC)\in T\left(q,d\right)$,
we need a ``relative sign'', denoted $\mathrm{Sgn}(\vR,\vC)$\marginpar{$\mathrm{Sgn}\left(\vR,\vC\right)$},
which will enable us to calculate the sign of every $\sigma\in\mathrm{Sym}(\vR,\vC)$
based solely on the signs of the permutation-matrix $\sigma$ restricted
to the minors $\left(R_{\ell},C_{\ell}\right)$. This is the sign
of the permutation-matrix obtained by assigning, for each $\ell$,
the identity matrix $I_{\left|R_{\ell}\right|}$ to the $\left(R_{\ell},C_{\ell}\right)$
minor. For example, if $\vR=\left(\left\{ 1,3,5\right\} ,\left\{ 2,4\right\} \right)$
and $\vC=\left(\left\{ 3,4,5\right\} ,\left\{ 1,2\right\} \right)$,
then 
\[
\mathrm{sgn}\left(\vR,\vC\right)=\mathrm{sgn}\left(\begin{array}{ccccc}
0 & 0 & 1 & 0 & 0\\
1 & 0 & 0 & 0 & 0\\
0 & 0 & 0 & 1 & 0\\
0 & 1 & 0 & 0 & 0\\
0 & 0 & 0 & 0 & 1
\end{array}\right).
\]

\begin{lem}
\label{lem:det-of-sum} If $A_{1},\ldots,A_{q}\in M_{d}\left(\mathbb{C}\right)$
are $d\times d$ matrices, then 
\[
|A_{1}+\cdots+A_{q}|=\sum_{(\vR,\vC)\in T\left(q,d\right)}\mathrm{sgn}\left(\vR,\vC\right)\prod_{\ell=1}^{q}\left|A_{\ell}\right|{}_{R_{\ell},C_{\ell}},
\]
where for $R,C\subseteq\left[d\right]$ with $\left|R\right|=\left|C\right|$,
\[
\left|A\right|_{R,C}=\begin{cases}
\det\left(\left(a_{i,j}\right){}_{i\in R,j\in C}\right) & \mbox{if }\left|R\right|=\left|C\right|\ge1\\
1 & \mbox{if }R=C=\emptyset
\end{cases}
\]
marks the determinant of the $\left(R,C\right)$-minor of $A$.
\end{lem}

\begin{proof}
For every $\sigma\in S_{d}$ and $s_{\sigma}:\left[d\right]\to\left[q\right]$
as in \eqref{eq:double-sum-for-det-of-sum}, the pair of partitions
$(\vR,\vC)$ satisfying $R_{\ell}=s_{\sigma}^{-1}\left(\ell\right)$
and $C_{\ell}=\sigma\left(R_{\ell}\right)$, for each $\ell$, is
the unique pair in $T\left(q,d\right)$ which respects $\sigma$ and
$s_{\sigma}$, and then $\sigma\in\mathrm{Sym}(\vR,\vC)$. Therefore,
\[
\left|A_{1}+\ldots+A_{q}\right|=\sum_{\left(\vR,\vC\right)\in T\left(q,d\right)}\,\sum_{\sigma\in\mathrm{Sym}\left(\vR,\vC\right)}\mathrm{sgn}\left(\sigma\right)\sum_{\begin{gathered}s_{\sigma}:\left[d\right]\to\left[q\right]\\
s_{\sigma}^{-1}\left(\ell\right)=R_{\ell}\,\,\forall\ell
\end{gathered}
}\prod_{i=1}^{d}\left[A_{s_{\sigma}\left(i\right)}\right]_{i,\sigma\left(i\right)}.
\]
Now, specifying a permutation $\sigma\in\mathrm{Sym}(\vR,\vC)$ is
equivalent to specifying the permutation $\sigma_{\ell}$ induced
by $\sigma$ on each of the minors $\left(R_{\ell},C_{\ell}\right)$.
Thus, $\mathrm{Sym}(\vR,\vC)\cong S_{\left|R_{1}\right|}\times\ldots\times S_{\left|R_{q}\right|}$
(as sets) via $\sigma\mapsto\left(\sigma_{1},\ldots,\sigma_{q}\right)$.
It is easy to see that the signs of these permutations are related
by 
\[
\mathrm{sgn}\left(\sigma\right)=\mathrm{sgn}\left(\vR,\vC\right)\cdot\mathrm{sgn}\left(\sigma_{1}\right)\cdot\ldots\cdot\mathrm{sgn}\left(\sigma_{q}\right).
\]
Thus, if $R_{\ell}\left(i\right)$ is the $i$-th element in $R_{\ell}$,
we get
\begin{eqnarray*}
\left|A_{1}+\ldots+A_{q}\right| & = & \sum_{\left(\vR,\vC\right)\in T\left(q,d\right)}\mathrm{sgn}\left(\vR,\vC\right)\prod_{\ell=1}^{q}\left[\sum_{\sigma_{\ell}\in S_{\left|R_{\ell}\right|}}\mathrm{sgn}\left(\sigma_{\ell}\right)\prod_{i=1}^{\left|R_{\ell}\right|}\left[A_{\ell}\right]_{R_{\ell}\left(i\right),C_{\ell}\left(\sigma_{\ell}\left(i\right)\right)}\right]\\
 & = & \sum_{\left(\vR,\vC\right)\in T\left(q,d\right)}\mathrm{sgn}\left(\vR,\vC\right)\prod_{\ell=1}^{q}\left|A_{\ell}\right|{}_{R_{\ell},C_{\ell}}.
\end{eqnarray*}
\end{proof}

\subsection{\label{subsec:matrix-coefs}Matrix Coefficients}

Recall that if $\pi$ is a $d$-dimensional representation, then $\bigwedge^{m}\pi$
is $\binom{d}{m}$-dimensional, and if $\left\{ v_{1},\ldots,v_{d}\right\} $
is a basis for $\mathbb{C}^{d}$, then $\left\{ v_{i_{1}}\wedge\ldots\wedge v_{i_{m}}\,\middle|\,1\le i_{1}<i_{2}<\ldots<i_{m}\le d\right\} $
is a basis for $\bigwedge^{m}\left(\mathbb{C}^{d}\right)$ (see Section
\ref{subsec:Group-Representations}). The following standard claim
and classical theorem explain the role in Theorem \ref{thm:P1} of
the conditions on $\bigwedge^{m}\pi$, for $0\le m\le d$, as defined
in property $\left({\cal P}1\right)$:
\begin{claim}[{E.g.~\cite[Theorem~6.6.3]{kung2009combinatorics}}]
\label{claim:mat-coefs-of-exterior-powers} If the matrices $\pi\left(g\right)$
are given in terms of the basis $V=\left\{ v_{1},\ldots,v_{d}\right\} $
and $\left(\bigwedge^{m}\pi\right)\left(g\right)$ in terms of the
basis $\left\{ v_{i_{1}}\wedge\ldots\wedge v_{i_{m}}\,\middle|\,1\le i_{1}<i_{2}<\ldots<i_{m}\le d\right\} $,
then the entry (matrix coefficient) of $\left(\bigwedge^{m}\pi\right)\left(g\right)$
in row $v_{i_{1}}\wedge\ldots\wedge v_{i_{m}}$ and column $v_{j_{1}}\wedge\ldots\wedge v_{j_{m}}$
is given by the minor-determinant $ $$\left|\pi\left(g\right)\right|_{\left\{ i_{1},\ldots,i_{r}\right\} ,\left\{ j_{1},\ldots,j_{r}\right\} }$.
\end{claim}

\begin{thm}[{Peter-Weyl, see e.g.~\cite[Chapter~2]{bump2004}}]
\label{thm:peter-weyl} The matrix coefficients of the irreducible
representations of a finite (compact) group $\Gamma$ are an orthogonal
basis of $L^{2}\left(\Gamma\right)$. More precisely, if $\pi_{1}\colon\Gamma\to U\left(d_{1}\right)$
and $\pi_{2}\colon\Gamma\to U\left(d_{2}\right)$ are irreducible
non-isomorphic unitary representations of $\Gamma$, then 
\[
\mathbb{E}_{g\in\Gamma}\left[\pi_{1}\left(g\right)_{i_{1},j_{1}}\cdot\overline{\pi_{2}\left(g\right)_{i_{2},j_{2}}}\right]=0
\]
for every $i_{1},j_{1}\in\left[d_{1}\right]$ and $i_{2},j_{2}\in\left[d_{2}\right]$,
the expectation taken according to the uniform (Haar, respectively)
measure of $\Gamma$. Moreover, if $\pi\colon\Gamma\to U\left(d\right)$
is an irreducible representation, then 
\[
\mathbb{E}_{g\in\Gamma}\left[\pi\left(g\right)_{i_{1},j_{1}}\cdot\overline{\pi\left(g\right)_{i_{2},j_{2}}}\right]=\begin{cases}
\frac{1}{d} & \left(i_{1},j_{1}\right)=\left(i_{2},j_{2}\right)\\
0 & otherwise
\end{cases}.
\]
\end{thm}

We now have all the tools needed to prove Theorem \ref{thm:P1}.

\subsection{\label{subsec:proof-of-thm P1}Proof of Theorem~\ref{thm:P1}}

Assume without loss of generality that $\pi\colon\Gamma\to U\left(d\right)$
maps the elements of $\Gamma$ to unitary matrices, so that for every
$e\in E\left(G\right)$, $A_{\gamma,\pi}\left(-e\right)$ is the conjugate-transpose
matrix $A_{\gamma,\pi}\left(e\right)^{*}$.

Recall \eqref{eq:A_gamma_pi as sigma over oriented edges}. We analyze
the expected characteristic polynomial 
\begin{equation}
\mathbb{E}_{\gamma\in{\cal C}_{\Gamma,G}}\left[\phi_{\gamma,\pi}\right]=\mathbb{E}_{\gamma\in{\cal C}_{\Gamma,G}}\left[\det\left(xI-A_{\gamma,\pi}\right)\right]=\mathbb{E}_{\gamma\in{\cal C}_{\Gamma,G}}\left[\det\left(xI-\sum_{e\in E\left(G\right)}A_{\gamma,\pi}\left(e\right)\right)\right].\label{eq:expected-char-poly}
\end{equation}
Our goal is to show it is equal to the formula given for $\M_{d,G}$
in Proposition \ref{prop:formula-for-M_d,G}. We use Lemma \ref{lem:det-of-sum}
to rewrite the determinant in the right hand side of \eqref{eq:expected-char-poly}.
We now let 
\[
T=T\left(1+\left|E\left(G\right)\right|,nd\right)=\left\{ \left(\vR,\vC\right)\,\middle|\,\begin{gathered}\vR\,\,\mathrm{and}\,\,\vC\,\,\mathrm{are\,partitions\,of}\,\left[nd\right]\,\,\mathrm{to}\,1+\left|E\left(G\right)\right|\,\mathrm{parts}\,\\
\mathrm{indexed\,by\:}\left\{ x\right\} \cup E\left(G\right),\\
\mathrm{with}\,R_{x}=C_{x}\,\mathrm{and}\,\left|R_{e}\right|=\left|C_{e}\right|\mathrm{for\,all}\,e\in E\left(G\right)
\end{gathered}
\right\} .
\]
By Lemma \ref{lem:det-of-sum}, 
\[
\phi_{\gamma,\pi}=\sum_{(\vR,\vC)\in T}\mathrm{sgn}\left(\vR,\vC\right)\cdot x^{\left|R_{x}\right|}\prod_{e\in E\left(G\right)}(-1)^{|R_{e}|}|A_{\gamma,\pi}\left(e\right)|_{R_{e},C_{e}}.
\]
Taking expected values gives
\begin{multline}
\mathbb{E}_{\gamma}\left[\phi_{\gamma,\pi}\right]=\sum_{(\vR,\vC)\in T}\mathrm{sgn}\left(\vR,\vC\right)\cdot x^{\left|R_{x}\right|}\mathbb{E}_{\gamma}\left[\prod_{e\in E\left(G\right)}(-1)^{|R_{e}|}|A_{\gamma,\pi}\left(e\right)|_{R_{e},C_{e}}\right]\\
=\sum_{(\vR,\vC)\in T}\mathrm{sgn}\left(\vR,\vC\right)\cdot x^{\left|R_{x}\right|}\left(-1\right)^{nd-\left|R_{x}\right|}\prod_{e\in E^{+}\left(G\right)}\mathbb{E}_{\gamma}\left[|A_{\gamma,\pi}\left(e\right)|_{R_{e},C_{e}}\cdot|A_{\gamma,\pi}\left(-e\right)|_{R_{-e},C_{-e}}\right],\label{eq:product_of_exp}
\end{multline}
since the $A_{\gamma,\pi}\left(e\right)$ are independent except for
the pairs $A_{\gamma,\pi}\left(e\right)$ and $A_{\gamma,\pi}\left(-e\right)$.

Since $A_{\gamma,\pi}\left(-e\right)=A_{\gamma,\pi}\left(e\right)^{*}$,
the term inside the expectation in the right hand side of \eqref{eq:product_of_exp}
is equal to 
\[
\mathbb{E}_{\gamma}\left[|A_{\gamma,\pi}\left(e\right)|_{R_{e},C_{e}}\cdot|A_{\gamma,\pi}\left(e\right)^{*}|_{R_{-e},C_{-e}}\right]=\mathbb{E}_{\gamma}\left[|A_{\gamma,\pi}\left(e\right)|_{R_{e},C_{e}}\cdot\overline{|A_{\gamma,\pi}\left(e\right)|_{C_{-e},R_{-e}}}\right].
\]
Clearly, this term is zero, unless the minors we choose for $e$ and
$-e$ are inside the $d\times d$ blocks corresponding to $e$ and
$-e$, respectively. That is, if $B_{v}$\marginpar{$B_{v}$} denotes
the set of $d$ indices of rows and columns corresponding to the vertex
$v\in V\left(G\right)$, then this term is zero unless $R_{e},C_{-e}\subseteq B_{t\left(e\right)}$
and $C_{e},R_{-e}\subseteq B_{h\left(e\right)}$ . If this is the
case, we can think of $R_{e},C_{e},R_{-e},C_{-e}$ as subsets of $\left[d\right]$,
so Claim \ref{claim:mat-coefs-of-exterior-powers} yields this term
is 
\[
\mathbb{E}_{\gamma}\left[\left(\left(\bigwedge\nolimits ^{\!\left|R_{e}\right|}\pi\right)\left(\gamma\left(e\right)\right)\right)_{R_{e},C_{e}}\cdot\overline{\left(\left(\bigwedge\nolimits ^{\!\left|R_{-e}\right|}\pi\right)\left(\gamma\left(e\right)\right)\right)_{C_{-e},R_{-e}}}\right],
\]
where we identify an $m$-subset of $\left[d\right]$ with a basis
element of $\bigwedge^{m}\left(\mathbb{C}^{d}\right)$ in the obvious
way. Finally, by the Peter-Weyl Theorem (Theorem \ref{thm:peter-weyl})
and our assumptions on the exterior powers $\bigwedge^{m}\pi$ for
$0\le m\le d$, this expectation is zero unless $\left|R_{e}\right|=\left|R_{-e}\right|$,
$R_{e}=C_{-e}$, and $C_{e}=R_{-e}$. If all these equalities hold,
the expectation is $\binom{d}{\left|R_{e}\right|}^{-1}$. 

Define $T^{\mathrm{sym}}\subseteq T$ to be the subset of $T$ containing
the partitions for which the expectation in \eqref{eq:product_of_exp}
is not zero. Namely, 
\[
T^{\mathrm{sym}}=\left\{ \left(\vR,\vC\right)\,\middle|\,\begin{gathered}\vR\,\,\mathrm{and}\,\,\vC\,\,\mathrm{are\,partitions\,of}\,\left[nd\right]\,\,\mathrm{to}\,\left|E\left(G\right)\right|+1\,\mathrm{parts}\,\\
\mathrm{indexed\,by\:}\left\{ x\right\} \cup E\left(G\right),\\
\mathrm{with}\,R_{x}=C_{x},\,\mathrm{and}\,\mathrm{for\,all}\,e\in E^{+}\left(G\right)\\
\left|R_{e}\right|=\left|C_{e}\right|,\,C_{-e}=R_{e},\,R_{-e}=C_{e},\,R_{e}\subseteq B_{t\left(e\right)}\,\mathrm{and}\,R_{-e}\subseteq B_{h\left(e\right)}
\end{gathered}
\right\} .
\]
Our discussion shows that
\[
\mathbb{E}_{\gamma}\left[\phi_{\gamma,\pi}\right]=\sum_{\left(\vR,\vC\right)\in T^{\mathrm{sym}}}\mathrm{sgn}(\vR,\vC)\cdot x^{\left|R_{x}\right|}\left(-1\right)^{nd-\left|R_{x}\right|}\prod_{e\in E^{+}\left(G\right)}\frac{1}{\binom{d}{\left|R_{e}\right|}}.
\]
Now, notice that because $\left|R_{-e}\right|=\left|R_{e}\right|$,
we get that $nd-\left|R_{x}\right|=\sum_{e\in E\left(G\right)}\left|R_{e}\right|$
is even, so $\left(-1\right)^{nd-\left|R_{x}\right|}=1$ for every
$(\vR,\vC)\in T^{\mathrm{sym}}$. Because of the conditions $C_{-e}=R_{e}$
and $R_{-e}=C_{e}$ on the partitions in $T^{\mathrm{sym}}$, the
permutation matrix defining $\mathrm{sgn}(\vR,\vC)$ is symmetric.
Thus, the corresponding permutation is an involution, with exactly
$\left|R_{x}\right|$ fixed points and $\frac{nd-\left|R_{x}\right|}{2}$
$2$-cycles\footnote{In particular, if $\left(\vR,\vC\right)\in T^{\mathrm{sym}}$ then
$R_{e}\cap C_{e}=\emptyset$, even for loops.}, so $\mathrm{sgn}(\vR,\vC)=\left(-1\right)^{\left(nd-\left|R_{x}\right|\right)/2}$.
Hence,
\[
\mathbb{E}_{\gamma}\left[\phi_{\gamma,\pi}\right]=\sum_{\left(\vR,\vC\right)\in T^{\mathrm{sym}}}\left(-1\right)^{\left(nd-\left|R_{x}\right|\right)/2}\cdot x^{\left|R_{x}\right|}\prod_{e\in E^{+}\left(G\right)}\frac{1}{\binom{d}{\left|R_{e}\right|}}.
\]

Recall the definition of a $d$-multi-matching given in Definition
\ref{def: d-multi-matching}: it is a function $m\colon E\left(G\right)\to\mathbb{Z}_{\ge0}$
such that $m\left(-e\right)=m\left(e\right)$ for every $e\in E\left(G\right)$
and $m\left(v\right)\le d$ for every $v\in V\left(G\right)$, where
$m\left(v\right)$ is the sum of $m$ on all oriented edges emanating
from $v$. For every $(\vR,\vC)\in T^{\mathrm{sym}}$, consider the
map $\eta(\vR,\vC)\colon E\left(G\right)\to\mathbb{Z}_{\ge0}$ given
by $e\mapsto|R_{e}|$. We claim this is a $d$-multi-matching. Indeed,
for every $v\in V\left(G\right)$,
\[
\eta\left(\vR,\vC\right)\left(v\right)=\sum_{e\in E\left(G\right):\,t\left(e\right)=v}\left|R_{e}\right|\le d
\]
as $\vR$ is a partition and $R_{e}\subseteq B_{t\left(e\right)}$
when $t\left(e\right)=v$.

Finally, for every $(\vR,\vC)\in T^{\mathrm{sym}}$, $\vC$ is completely
determined by $\vR$. Denote by $e_{v,1},\ldots,e_{v,\deg\left(v\right)}$
the oriented edges emanating from $v$. Then, for every $d$-multi-matching
$m$, the number of partitions $(\vR,\vC)\in T^{\mathrm{sym}}$ associated
to $m$ is exactly 
\[
\prod_{v\in V\left(G\right)}\binom{d}{m\left(e_{v,1}\right),\ldots,m\left(e_{v,\deg\left(v\right)}\right)}.
\]
We obtain
\begin{eqnarray*}
\mathbb{E}_{\gamma}\left[\phi_{\gamma,\pi}\right] & = & \sum_{m\in\multi_{d}\left(G\right)}\,\sum_{(\vR,\vC)\in T^{\mathrm{sym}}\,:\,\eta(\vR,\vC)=m}\left(-1\right)^{\left(nd-\left|R_{x}\right|\right)/2}\cdot x^{\left|R_{x}\right|}\prod_{e\in E^{+}\left(G\right)}\frac{1}{\binom{d}{\left|R_{e}\right|}}\\
 & = & \sum_{m\in\multi_{d}\left(G\right)}\left(-1\right)^{\left|m\right|}x^{nd-2\left|m\right|}\frac{{\displaystyle \prod_{v\in V\left(G\right)}\binom{d}{m\left(e_{v,1}\right),\ldots,m\left(e_{v,\deg\left(v\right)}\right)}}}{{\displaystyle \prod_{e\in E^{+}\left(G\right)}\binom{d}{m\left(e\right)}}},
\end{eqnarray*}
where the summation is over all $d$-multi-matchings $m$ of $G$,
and $\left|m\right|=\sum_{e\in E^{+}\left(G\right)}m\left(e\right)$.
This is precisely the formula for $\M_{d,G}$ from Proposition \ref{prop:formula-for-M_d,G},
so the proof of Theorem \ref{thm:P1} is complete. \qed 
\begin{rem}
To give further intuition for Theorem \ref{thm:P1}, we remark that
its statement for one particular family of pairs $\left(\Gamma,\pi\right)$
satisfying $\left({\cal P}1\right)$ follows readily from well known
results. This is the family of signed permutation groups: for every
$d$, let $\pi\left(\Gamma\right)$ be the subgroup of $\mathrm{GL}_{d}\left(\mathbb{C}\right)$
consisting of matrices with entries in $\left\{ 0,\pm1\right\} $
and with exactly one non-zero entry in every row and in every column.
That these pairs satisfy $\left({\cal P}1\right)$ (and also $\left({\cal P}2\right)$)
is well known - see Section \ref{sec:applications}. But here, every
$\left(\Gamma,\pi\right)$-covering of a graph $G$ corresponds to
a $d$-covering $H$ plus a signing (a $\left(\mathbb{Z}/2\mathbb{Z},\mathrm{sign}\right)$-covering)
of $H$. The special case of Theorem \ref{thm:P1} for $\left(\mathbb{Z}/2\mathbb{Z},\mathrm{sign}\right)$-coverings,
known at least since \cite{GG81}, shows that the expected characteristic
polynomial over all signings of $H$ equals $\M_{1,H}$, hence the
average of all signings over all possible $d$-coverings $H$ is $\M_{d,G}$,
by definition. This gives a shorter route to proving Theorem \ref{thm:M_d,G is Ramanujan}.
\end{rem}

\section{Property $\left({\cal P}2\right)$ and the Proof of Theorem \ref{thm:P2}\label{sec:P2}}

The main goal of this Section is to show that the expected characteristic
polynomial of a random $\left(\Gamma,\pi\right)$-covering is real
rooted for certain distributions on coverings, and in particular that
when $\left(\Gamma,\pi\right)$ satisfies $\left(\P2\right)$, this
is true for the uniform distribution. The main component of the proof
is Theorem \ref{thm:MSS15}, showing that for certain distributions
of Hermitian ($A^{*}=A$) matrices, the expected characteristic polynomial
is real rooted. This theorem imitates and generalizes the argument
of \cite[Theorem 3.3]{MSS15}. We repeat the argument in Section \ref{subsec:MSS15},
because we need a more general statement, but we refer the interested
reader to \cite[Section 3]{MSS15} for some more elaborated concepts
and notions. Theorem \ref{thm:MSS15} is a generalization of the fact
that the characteristic polynomials\footnote{Recall that for any matrix $A$, we denote its characteristic polynomial
by $\phi\left(A\right)=\det\left(xI-A\right)$.} $\phi\left(A\right)$ and $\phi\left(BAB^{*}\right)$ interlace whenever
$A\in\mathrm{M}_{d}\left(\mathbb{C}\right)$ is Hermitian and $B\in U\left(d\right)$
satisfies $\mathrm{rank}\left(B-I_{d}\right)=1$. 

\subsection{Average Characteristic Polynomial of Sum of Random Matrices\label{subsec:MSS15}}
\begin{defn}
\label{def:rank-1-variable}We say that the random variable $W$ taking
values in $U\left(d\right)$ is\textbf{ }a\textbf{ rank-1 random variable}
if every two different possible values $B_{1}$ and $B_{2}$ satisfy
$\mathrm{rank}\left(B_{1}B_{2}^{-1}-I_{d}\right)=1$.
\end{defn}

It is not hard to see that $W$ is a rank-1 random variable if and
only if it takes values in some $P\Lambda Q$ where $P,Q\in U\left(d\right)$
and \marginpar{$\Lambda$}$\Lambda\le U\left(d\right)$ is the subgroup
of diagonal matrices 
\[
\left\{ \left(\begin{array}{cccc}
\lambda\\
 & 1\\
 &  & \ddots\\
 &  &  & 1
\end{array}\right)\in U\left(d\right)\,\middle|\,\left|\lambda\right|=1\right\} .
\]

\begin{thm}
\label{thm:MSS15}Let $m\in\mathbb{Z}_{\ge1}$, let $\ell\left(1\right),\ldots,\ell\left(m\right)\in\mathbb{Z}_{\ge0}$,
and let ${\cal W}=\left\{ W_{i,j}\right\} _{1\le i\le m,1\le j\le\ell\left(i\right)}$
be a set of independent rank-1 random variables taking values in $U\left(d\right)$.
If $A_{1},\ldots,A_{m}\in\mathrm{M}_{d}\left(\mathbb{C}\right)$ are
Hermitian matrices, then the expected characteristic polynomial
\[
P_{{\cal W}}\left(A_{1},\ldots,A_{m}\right)\overset{\mathrm{def}}{=}\mathbb{E}_{{\cal W}}\left[\phi\left(\sum_{i=1}^{m}W_{i,1}\ldots W_{i,\ell\left(i\right)}A_{i}W_{i,\ell\left(i\right)}^{*}\ldots W_{i,1}^{*}\right)\right]
\]
is real rooted.
\end{thm}

Note that the characteristic polynomial of a Hermitian matrix is in
$\mathbb{R}\left[x\right]$, and so $P_{{\cal W}}\left(A_{1},\ldots,A_{m}\right)$,
which is an average of such polynomials, is also in $\mathbb{R}\left[x\right]$.
\begin{lem}
\label{lem:rank-1 linearity}In the notation of Theorem \ref{thm:MSS15},
assume that $P_{{\cal W}}\left(A_{1},\ldots,A_{m}\right)$ is real
rooted whenever $A_{1},\ldots,A_{m}\in\mathrm{M}_{d}\left(\mathbb{C}\right)$
are Hermitian. Then, for every $v\in\mathbb{C}^{d}$,\textup{ the
roots $\alpha_{d}\le\ldots\le\alpha_{1}$ of $P_{{\cal W}}\left(A_{1},\ldots,A_{i}+vv^{*},\ldots,A_{m}\right)$
and the roots $\beta_{d}\le\ldots\le\beta_{1}$ of $P_{{\cal W}}\left(A_{1},\ldots,A_{m}\right)$
satisfy
\[
\beta_{d}\le\alpha_{d}\le\beta_{d-1}\le\alpha_{d-1}\le\ldots\le\beta_{1}\le\alpha_{1}.
\]
In other words, the polynomials $P_{{\cal W}}\left(A_{1},\ldots,A_{i}+vv^{*},\ldots,A_{m}\right)$
and $P_{{\cal W}}\left(A_{1},\ldots,A_{m}\right)$ interlace in a
strong sense.}
\end{lem}

\begin{proof}
Denote $Q\left(A_{1},\ldots,A_{m}\right)=\left[\frac{\partial}{\partial\mu}P_{{\cal W}}\left(A_{1},A_{2},\ldots,A_{i}+\mu vv^{*},\ldots,A_{m}\right)\right]_{\mu=0}$.
We claim that\footnote{This property is called ``Rank-1 linearity'' in \cite{MSS15}.}
\begin{equation}
P_{{\cal W}}\left(A_{1},\ldots,A_{i}+\mu vv^{*},\ldots,A_{m}\right)=P_{{\cal W}}\left(A_{1},\ldots,A_{m}\right)+\mu\cdot Q\left(A_{1},\ldots,A_{m}\right)\label{eq:rank-1 linearity}
\end{equation}
for every $\mu\in\mathbb{C}$. To see this, it is enough to show \eqref{eq:rank-1 linearity}
in the case ${\cal W}$ is constant (that is, $W_{i,j}$ is constant
for every $i,j$), the general statement will then follow by linearity
of expectation and of taking the derivative. For a constant ${\cal W}$,
we only need to show that for any Hermitian $A\in\mathrm{M}_{d}\left(\mathbb{C}\right)$,
the characteristic polynomial $\phi\left(A+\mu vv^{*}\right)$ is
linear in $\mu$. By conjugating $A$ and $vv*$ by some unitary matrix,
we may assume $v^{*}=\left(\alpha,0,0,\ldots,0\right)$, and then
the claim is clear by expanding the determinant of $xI-\left(A+\mu vv^{*}\right)$
along, say, the first row. 

If $A_{i}$ is Hermitian, then for every $\mu\in\mathbb{R}$, $A_{i}+\mu vv^{*}$
is Hermitian, so by our assumption the left hand side of \eqref{eq:rank-1 linearity}
is real rooted. Therefore, the right hand side $P_{{\cal W}}\left(A_{1},\ldots,A_{m}\right)+\mu\cdot Q\left(A_{1},\ldots,A_{m}\right)$
is also real rooted for every $\mu\in\mathbb{R}$. Note that while
$P_{{\cal W}}\left(A_{1},\ldots,A_{m}\right)$ is monic of degree
$d$ (in the variable $x$), $Q\left(A_{1},\ldots,A_{m}\right)$ is
a polynomial of degree $d-1$ with a negative leading coefficient
(because $\left[\frac{\partial}{\partial\mu}\det\left(xI-\left(A+\mu vv^{*}\right)\right)\right]_{\mu=0}$
has this property when ${\cal W}$ is constant ). It follows from
Fact \ref{fact:interlace} that $P_{{\cal W}}\left(A_{1},\ldots,A_{m}\right)$
and $Q\left(A_{1},\ldots,A_{m}\right)$ interlace, and the real roots
$\vartheta_{d-1}\le\ldots\le\vartheta_{1}$ of $Q\left(A_{1},\ldots,A_{m}\right)$
satisfy
\[
\beta_{d}\le\vartheta_{d-1}\le\beta_{d-1}\le\ldots\le\beta_{2}\le\vartheta_{1}\le\beta_{1}.
\]
Moreover, the $i$-th root of $P_{{\cal W}}\left(A_{1},\ldots,A_{m}\right)+\mu\cdot Q\left(A_{1},\ldots,A_{m}\right)$
grows continuously from $\beta_{i}$ to $\vartheta_{i-1}$ as $\mu$
grows (the largest root grows from $\beta_{1}$ to $+\infty$). In
particular, $\mu=1$ corresponds to $P_{{\cal W}}\left(A_{1},\ldots,A_{i}+vv^{*},\ldots,A_{m}\right)$,
so the roots $\alpha_{d}\le\ldots\le\alpha_{1}$ satisfy
\[
\beta_{d}\le\alpha_{d}\le\vartheta_{d-1}\le\beta_{d-1}\le\alpha_{d-1}\le\vartheta_{d-1}\le\ldots\le\beta_{2}\le\alpha_{2}\le\vartheta_{1}\le\beta_{1}\le\alpha_{1}.
\]
\end{proof}

\begin{proof}[Proof of Theorem \ref{thm:MSS15}]
 We prove by induction on the number of $W_{i,j}$'s, namely, induction
on $\ell\left({\cal W}\right)=\ell\left(1\right)+\ldots+\ell\left(m\right)$.
The statement is clear for $\ell\left({\cal W}\right)=0$. Given ${\cal W}$
with $\ell\left(W\right)>0$, assume without loss of generality that
$\ell\left(1\right)>0$, and denote by ${\cal W}'$ the set of random
variables ${\cal W}\smallsetminus\left\{ W_{1,\ell\left(1\right)}\right\} $.
We assume, by the induction hypothesis, that $P_{{\cal W}'}\left(A_{1},\ldots,A_{m}\right)$
is real rooted for all Hermitian matrices $A_{1},\ldots,A_{m}$.

For clarity we assume $W_{1,\ell\left(1\right)}$ takes only finitely
many values, but the same argument works in the general case (Fact
\ref{fact:common interlacing} has a variant for a set of polynomials
of any cardinality). Let $B_{1},\ldots,B_{t}\in U\left(d\right)$
be the possible values of $W_{1,\ell\left(1\right)}$, obtained with
respective probabilities $p_{1},\ldots,p_{t}$. We need to show that
\[
P_{{\cal W}}\left(A_{1},\ldots,A_{m}\right)=p_{1}\cdot P_{{\cal W}'}\left(B_{1}A_{1}B_{1}^{*},A_{2},\ldots,A_{m}\right)+\cdots+p_{t}\cdot P_{{\cal W}'}\left(B_{t}A_{1}B_{t}^{*},A_{2},\ldots,A_{m}\right)
\]
is real rooted for all Hermitian matrices $A_{1},\ldots,A_{m}\in M_{d}\left(\mathbb{C}\right)$.
By Fact \ref{fact:common interlacing}, it is enough to show the $t$
polynomials $P_{{\cal W}'}(B_{j}A_{1}B_{j}^{*},A_{2},\ldots,A_{m})$
$\left(1\le j\le t\right)$ have a common interlacing. By definition,
this is equivalent to showing that any two of them have a common interlacing.
Thus, it is enough to show that if $B,C\in U\left(d\right)$ satisfy
$\mathrm{rank}\left(BC^{-1}-I_{d}\right)=1$, then the polynomials
$P_{{\cal W}'}\left(BA_{1}B^{*},A_{2},\ldots,A_{m}\right)$ and $P_{{\cal W}'}\left(CA_{1}C^{*},A_{2},\ldots,A_{m}\right)$
have common interlacing.

By replacing $A_{1}$ with $CA_{1}C^{*}$ and writing $D=BC^{-1}$
we need to prove that $P_{{\cal W}'}\left(DA_{1}D^{*},A_{2},\ldots,A_{m}\right)$
and $P_{{\cal W}'}\left(A_{1},A_{2},\ldots,A_{m}\right)$ have a common
interlacing (now $\mathrm{rank}\left(D-I_{d}\right)=1$). If $DA_{1}D^{*}=A_{1}$
the statement is obvious. Otherwise, we claim that $DA_{1}D^{*}-A_{1}$
is a rank-2 trace-0 Hermitian matrix: by unitary conjugation we may
assume 
\[
D=\left(\begin{array}{cccc}
\lambda\\
 & 1\\
 &  & \ddots\\
 &  &  & 1
\end{array}\right)\in\Lambda,
\]
and direct calculation then shows that 
\[
DA_{1}D^{*}-A_{1}=\left(\begin{array}{cccc}
0 & \lambda a_{1,2} & \cdots & \lambda a_{1,d}\\
\overline{\lambda}a_{2,1} & 0 & \cdots & 0\\
\vdots & \vdots & \ddots & \vdots\\
\overline{\lambda}a_{d,1} & 0 & \cdots & 0
\end{array}\right).
\]
Let $\pm\nu$ ($\nu\in\mathbb{R}_{>0}$) be the non-zero eigenvalues
of $DA_{1}D^{*}-A_{1}$. By spectral decomposition, $DA_{1}D^{*}-A_{1}=uu^{*}-vv^{*}$
for some vectors $u,v\in\mathbb{C}^{d}$ of length $\sqrt{\nu}$.
Consider also $P_{{\cal W}'}\left(A_{1}-vv^{*},A_{2},\ldots,A_{m}\right)$
and denote
\begin{eqnarray*}
\alpha_{d}\le\ldots\le\alpha_{1} &  & \mathrm{the\,roots\,of\;}P_{{\cal W}'}\left(A_{1},A_{2},\ldots,A_{m}\right)\\
\beta_{d}\le\ldots\le\beta_{1} &  & \mathrm{the\,roots\,of\;}P_{{\cal W}'}\left(A_{1}-vv^{*},A_{2},\ldots,A_{m}\right)\\
\gamma_{d}\le\ldots\le\gamma_{1} &  & \mathrm{the\,roots\,of\;}P_{{\cal W}'}\left(DA_{1}D^{*},A_{2},\ldots,A_{m}\right).
\end{eqnarray*}
The assumptions of Lemma \ref{lem:rank-1 linearity} are satisfied
for ${\cal W}'$ by the induction hypothesis. We can apply this lemma
on $P_{{\cal W}'}\left(A_{1}-vv^{*},A_{2},\ldots,A_{m}\right)$ and
$P_{{\cal W}'}\left(A_{1},A_{2},\ldots,A_{m}\right)$ to obtain that
\[
\beta_{d}\le\alpha_{d}\le\beta_{d-1}\le\alpha_{d-1}\le\ldots\le\beta_{2}\le\alpha_{2}\le\beta_{1}\le\alpha_{1}.
\]
Similarly, we can apply the lemma on 
\[
P_{{\cal W}'}\left(A_{1}-vv^{*},A_{2},\ldots,A_{m}\right)\,\,\mathrm{and}\,\,P_{{\cal W}'}\left(DA_{1}D^{*},A_{2},\ldots,A_{m}\right)=P_{{\cal W}'}\left(A_{1}-vv^{*}+uu*,A_{2},\ldots,A_{m}\right)
\]
 to obtain that
\[
\beta_{d}\le\gamma_{d}\le\beta_{d-1}\le\gamma_{d-1}\le\ldots\le\beta_{2}\le\gamma_{2}\le\beta_{1}\le\gamma_{1}.
\]
It follows that $P_{{\cal W}'}\left(DA_{1}D^{*},A_{2},\ldots,A_{m}\right)$
and $P_{{\cal W}'}\left(A_{1},A_{2},\ldots,A_{m}\right)$ have a common
interlacing. This completes the proof.
\end{proof}

\subsection{Average Characteristic Polynomial of Random Coverings}

Let $G$ be a finite graph \emph{without loops}, $\Gamma$ a group
and $\pi\colon\Gamma\to\mathrm{GL_{d}}\left(\mathbb{C}\right)$ a
unitary representation. We now deduce from Theorem \ref{thm:MSS15}
that for certain distributions of $\left(\Gamma,\pi\right)$-coverings
of $G$, the average characteristic polynomial is real rooted. Recall
that $\phi_{\gamma,\pi}$ denotes the characteristic polynomial of
$A_{\gamma,\pi}$ -- see \eqref{eq:def-of-phi}. 
\begin{prop}
\label{prop:avg-char-poly-of-random-coverings-is-real-rooted}Let
$X_{1},\ldots,X_{r}$ be independent random variables, each taking
values in the space of $\Gamma$-labelings of $G$. Suppose that all
possible values of $X_{i}$ agree on all edges in $E^{+}\left(G\right)$
except (possibly) for one, and on that edge suppose that $\pi\left(X\left(e\right)\right)$
is a rank-1 random variable of matrices in $U\left(d\right)$, as
in Definition \ref{def:rank-1-variable}. Then $\mathbb{E}_{X_{1}\cdots X_{r}}\left[\phi_{X_{1}\cdot\ldots\cdot X_{r},\pi}\right]$
is real rooted.
\end{prop}

\begin{proof}
As we noted in the proof of Claim \ref{claim: real spectrum}, for
any $\Gamma$-labeling $\gamma$, we have $\phi_{\gamma,\pi}=\phi_{\gamma,\pi'}$
whenever $\pi$ and $\pi'$ are isomorphic, so we assume without loss
of generality that $\pi\left(\Gamma\right)\subseteq U\left(d\right)$.
For every $\Gamma$-labeling $\gamma$, the matrix $A_{\gamma,\pi}$
is a $nd\times nd$ matrix composed of $n^{2}$ blocks of size $d\times d$.
The blocks are indexed by ordered pairs of vertices of $G$. Similarly
to a notation we used on Page \pageref{notation: A_gamma,pi}, for
any $e\in E^{+}\left(G\right)$, we let $A_{\gamma,\pi}^{\pm}\left(e\right)\in\mathrm{M}_{nd}$\marginpar{$A_{\gamma,\pi}^{\pm}\left(e\right)$}
be the matrix with zero blocks except for the blocks corresponding
to $e$ and to $-e$. In the block $\left(t\left(e\right),h\left(e\right)\right)$
we have $\pi\left(\gamma\left(e\right)\right)$ and in the block $\left(h\left(e\right),t\left(e\right)\right)$
we have $\pi\left(\gamma\left(-e\right)\right)=\pi\left(\gamma\left(e\right)\right)^{*}$.
 It is clear that $A_{\gamma,\pi}^{\pm}\left(e\right)$ is Hermitian
and that
\[
A_{\gamma,\pi}=\sum_{e\in E^{+}\left(G\right)}A_{\gamma,\pi}^{\pm}\left(e\right).
\]
For every random $\Gamma$-labeling $X$ of $G$ and $e\in E^{+}\left(G\right)$,
let $W_{e}\left(X\right)$ be the following random matrix in $U\left(nd\right)$:
\[
W_{e}\left(X\right)=\left(\begin{array}{ccccc}
I_{d}\\
 & \ddots\\
 &  & \pi\left(X\left(e\right)\right)\\
 &  &  & \ddots\\
 &  &  &  & I_{d}
\end{array}\right)
\]
where all non-diagonal $d\times d$ blocks are zeros, the $\left(t\left(e\right),t\left(e\right)\right)$
block is $\pi\left(X\left(e\right)\right)$, and the remaining diagonal
blocks are $I_{d}$. Also let $\mathbf{1}\colon E\left(G\right)\to\Gamma$
be the trivial labeling which labels all edges by the identity element
of $\Gamma$. With these notations, we have\footnote{This formula is exactly the place this proof breaks for loops.}
\[
A_{X_{1}\cdot\ldots\cdot X_{r},\pi}^{\pm}\left(e\right)=W_{e}\left(X_{1}\right)\cdot\ldots\cdot W_{e}\left(X_{r}\right)A_{\mathbf{1},\pi}^{\pm}\left(e\right)W_{e}\left(X_{r}\right)^{*}\cdot\ldots\cdot W_{e}\left(X_{1}\right)^{*},
\]
and
\[
A_{X_{1}\cdot\ldots\cdot X_{r},\pi}=\sum_{e\in E^{+}\left(G\right)}A_{X_{1}\cdot\ldots\cdot X_{r},\pi}^{\pm}\left(e\right).
\]
By assumption, the random $\Gamma$-labeling $X_{i}$ is constant
on all edges except for on one edge $e$. Thus, $\left\{ X_{i}\left(e\right)\right\} _{e\in E^{+}\left(G\right)}$
is a set of independent variables. Moreover, the set $\left\{ W_{e}\left(X_{i}\right)\right\} _{e\in E^{+}\left(G\right),1\le i\le r}$
is a set of independent variables taking values in $U\left(nd\right)$,
and every $W_{e}\left(X_{i}\right)$ is a rank-1 random variable by
the hypothesis on the values of $X_{i}$. The proposition now follows
by applying Theorem \ref{thm:MSS15}.
\end{proof}
Repeating the line of argument we explained in Section \ref{subsec:Overview-of-the},
we deduce:
\begin{cor}
\label{cor:actual-labeling-beating-average}In the notation of Proposition
\ref{prop:avg-char-poly-of-random-coverings-is-real-rooted}, there
is a $\Gamma$-labeling $\gamma=\gamma_{1}\cdots\gamma_{r}$ of $G$,
with $\gamma_{i}$ in the support of $X_{i}$, such that the largest
root of $\phi_{\gamma,\pi}$ is at most the largest root of $\mathbb{E}_{X_{1}\cdots X_{r}}\left[\phi_{X_{1}\cdots X_{r},\pi}\right]$.
\end{cor}

\begin{proof}
We prove by induction on $r$. Consider the support $\mathrm{Supp}\left(X_{1}\right)\subseteq\left\{ \Gamma\mathrm{-labelings\,of}\,G\right\} $.
Note that if $X_{1},\ldots,X_{r}$ satisfy the assumptions of Proposition
\ref{prop:avg-char-poly-of-random-coverings-is-real-rooted}, then
so would $X_{1}',X_{2},\ldots,X_{r}$ when $X_{1}'$ is any random
variable with the same support as $X_{1}$. Therefore, by Fact \ref{fact:common interlacing},
all polynomials in the family 
\[
\left\{ \mathbb{E}_{X_{2}\cdots X_{r}}\left[\phi_{\gamma\cdot X_{2}\cdots X_{r},\pi}\right]\right\} _{\gamma\in\mathrm{Supp}\left(X_{1}\right)}
\]
have a common interlacing. In particular, there is a polynomial in
this family, say the one defined by $\gamma_{1}\in\mathrm{Supp}\left(X_{1}\right)$,
with maximal root at most the maximal root of $\mathbb{E}_{X_{1}\cdots X_{r}}\left[\phi_{X_{1}\cdots X_{r},\pi}\right]$.
If $r=1$ we are done. If $r\ge2$, define $X_{2}'=\gamma_{1}\cdot X_{2}$.
The random variables $X_{2}',X_{3},\ldots,X_{r}$ still satisfy the
hypotheses of Proposition \ref{prop:avg-char-poly-of-random-coverings-is-real-rooted}.
Hence, by the induction hypothesis there is a $\Gamma$-labeling $\gamma=\gamma_{2}'\cdot\gamma_{3}\cdots\gamma_{r}$
of $G$ with $\gamma_{2}'\in\mathrm{Supp}\left(X_{2}'\right)$ and
$\gamma_{i}\in\mathrm{Supp}\left(X_{i}\right)$ for $3\le i\le r$,
such that the largest root of $\phi_{\gamma,\pi}$ is at most the
largest root of $\mathbb{E}_{X_{2}'\cdot X_{3}\cdots X_{r}}[\phi_{X_{2}'\cdot X_{3}\cdots X_{r},\pi}]$,
which in turn is at most the largest root of $\mathbb{E}_{X_{1}\cdots X_{r}}[\phi_{X_{1}\cdots X_{r},\pi}]$.
The statement of the corollary is now satisfied with $\gamma_{1},\gamma_{1}^{-1}\gamma_{2}',\gamma_{3},\ldots,\gamma_{r}$.
\end{proof}

\subsection{\label{subsec:Proof-of-Theorem P2}Proof of Theorem \ref{thm:P2}}

We finally have all the tools needed to prove Theorem \ref{thm:P2}.
Let $G$ be a finite, loopless graph, and let $\left(\Gamma,\pi\right)$
satisfy property $\left({\cal P}2\right)$, that is, $\Gamma$ is
a finite group, $\pi\colon G\to\mathrm{GL}_{d}\left(\mathbb{C}\right)$
is a representation and $\pi\left(\Gamma\right)$ is a complex reflection
group (i.e., generated by pseudo-reflections). Assume that $\Gamma$
is generated by $g_{1},\ldots,g_{s}$, where $\pi\left(g_{i}\right)$
is a pseudo-reflection for all $i$, i.e., $\mathrm{rank}\left(\pi\left(g_{i}\right)-I_{d}\right)=1$.
We first show that a certain lazy random walk on $\Gamma$, where
in each step we use only one of the $g_{i}$'s, converges to the uniform
distribution:
\begin{claim}
\label{claim: random-walk-converges}Define a random walk $\left\{ a_{n}\right\} _{n=0}^{\infty}$
on $\Gamma$ as follows: $a_{0}=\mathbf{1}_{\Gamma}$ (the identity
element of $\Gamma$), and for $n\ge1$ 
\begin{eqnarray*}
a_{n} & = & \begin{cases}
g_{n\,\mathrm{mod}\,s}\cdot a_{n-1} & \mathrm{with\,probability}\,\frac{1}{3}\\
\left(g_{n\,\mathrm{mod}\,s}\right)^{-1}\cdot a_{n-1} & \mathrm{with\,probability}\,\frac{1}{3}\\
a_{n-1} & \mathrm{with\,probability}\,\frac{1}{3}
\end{cases}.
\end{eqnarray*}
Then $a_{1},a_{2},\ldots$ converges to the uniform distribution on
$\Gamma$.
\end{claim}

\begin{proof}
Consider $a_{n}$ as an element of the group-ring $\mathbb{C}\left[\Gamma\right]$
so that the coefficient of $g$ is $\mathrm{Prob}\left[a_{n}=g\right]$.
Then for $n\ge1$, 
\[
a_{s\cdot n}=\left(\frac{1}{3}\mathbf{1}_{\Gamma}+\frac{1}{3}g_{s}+\frac{1}{3}g_{s}^{-1}\right)\cdots\left(\frac{1}{3}\mathbf{1}_{\Gamma}+\frac{1}{3}g_{1}+\frac{1}{3}g_{1}^{-1}\right)\cdot a_{s\cdot\left(n-1\right)}.
\]
The $s$-steps random walk $\left\{ a_{s\cdot n}\right\} _{n=0}^{\infty}$
is defined by the distribution 
\[
h=\left(\frac{1}{3}\mathbf{1}_{\Gamma}+\frac{1}{3}g_{s}+\frac{1}{3}g_{s}^{-1}\right)\cdots\left(\frac{1}{3}\mathbf{1}_{\Gamma}+\frac{1}{3}g_{1}+\frac{1}{3}g_{1}^{-1}\right),
\]
which satisfies $\mathrm{supp}\left(h^{n}\right)=\Gamma$ for all
large $n$. Thus, it converges to the only stationary distribution
of this Markov chain: the uniform distribution. The same argument
applies to $\left\{ a_{s\cdot n+i}\right\} _{n=0}^{\infty}$ for any
remainder $1\le i\le s-1$.
\end{proof}
Now define random $\Gamma$-labelings $\left\{ Z_{n}\right\} _{n=1}^{\infty}$
of $G$ as follows: let $\varepsilon=\left|E^{+}\left(G\right)\right|$
and enumerate the edges of $G$ in an arbitrary order, so $E^{+}\left(G\right)=\left\{ e_{1},\ldots,e_{\varepsilon}\right\} $.
For $i\ge1$ and $1\le j\le\varepsilon$ define $X_{i,j}$ to be the
random $\Gamma$-labeling of $G$ which labels every edge besides
$e_{j}$ with the identity element $\mathbf{1}_{\Gamma}$, and 
\begin{eqnarray*}
X_{i,j}\left(e_{j}\right) & = & \begin{cases}
g_{i\,\mathrm{mod}\,s} & \mathrm{with\,probability}\,\frac{1}{3}\\
\left(g_{i\,\mathrm{mod}\,s}\right)^{-1} & \mathrm{with\,probability}\,\frac{1}{3}\\
\mathbf{1}_{\Gamma} & \mathrm{with\,probability}\,\frac{1}{3}
\end{cases}.
\end{eqnarray*}
Now define $Y_{i}=X_{i,1}\cdots X_{i,\varepsilon}$ and $Z_{n}=Y_{1}Y_{2}\cdots Y_{n}$.
By definition, each random $\Gamma$-labeling $X_{i,j}$ is constant
on every edge except one, and on the remaining edge the ratio of every
two values is a pseudo-reflection. Proposition \ref{prop:avg-char-poly-of-random-coverings-is-real-rooted}
yields, therefore, that $\mathbb{E}_{Z_{n}}\left[\phi_{Z_{n},\pi}\right]$
is real rooted. By Claim \ref{claim: random-walk-converges}, the
random $\Gamma$-labelings $Z_{n}$ converge, as $n\to\infty$, to
the uniform distribution ${\cal C}_{\Gamma,G}$ of all $\Gamma$-labelings
of $G$. Since the map $Z\to\mathbb{E}_{Z}\left[\phi_{Z,\pi}\right]$
is a continuous map from the space of distributions of $\Gamma$-labelings
of $G$ to $\mathbb{R}\left[x\right]$, we get that $\mathbb{E}_{\gamma\in{\cal C}_{\Gamma,G}}\left[\phi_{\gamma,\pi}\right]$
is real rooted, thus the first statement of Theorem \ref{thm:P2}
holds.

Finally, by Corollary \ref{cor:actual-labeling-beating-average},
for every $n$, there is a $\Gamma$-labeling $\gamma_{n}$ of $G$
so that the largest root of $\phi_{\gamma,\pi}$ is at most the largest
root of $\mathbb{E}_{Z_{n}}\left[\phi_{Z_{n},\pi}\right]$. Because
the set of $\Gamma$-labeling of $G$ is finite, the $\gamma_{n}$
have an accumulation point $\gamma_{0}$. As the largest root of $\mathbb{E}_{Z_{n}}\left[\phi_{Z_{n},\pi}\right]$
converges, as $n\to\infty$, to the largest root of $\mathbb{E}_{\gamma\in{\cal C}_{\Gamma,G}}\left[\phi_{\gamma,\pi}\right]$,
the largest root of $\phi_{\gamma_{0},\pi}$ is at most the largest
root of $\mathbb{E}_{\gamma\in{\cal C}_{\Gamma,G}}\left[\phi_{\gamma,\pi}\right]$.
This completes the proof of Theorem \ref{thm:P2}.
\begin{rem}
\label{remark:middle point for Sym group}When $\Gamma=S_{r}$ is
the symmetric group, \cite[Lemma 3.5]{MSS15} gives a specific sequence
of $2^{r-1}-1$ rank-1 random permutations (``random swaps'' in
their terminology) the product of which is the uniform distribution
on $S_{d}$. In fact, $\binom{r}{2}$ random swap are enough. This
can be seen by the following inductive construction: let $X$ be a
uniformly random permutation in $S_{r-1}\le S_{r}$, and define 
\[
Y_{1}=\begin{cases}
\left(1\,r\right) & \frac{1}{r}\\
\mathrm{id} & \frac{r-1}{r}
\end{cases},\,\,\,\,\,Y_{2}=\begin{cases}
\left(2\,r\right) & \frac{1}{r-1}\\
\mathrm{id} & \frac{r-2}{r-1}
\end{cases},\,\,\,\,\,\ldots,\,\,\,\,\,Y_{r-1}=\begin{cases}
\left(r-1\,r\right) & \frac{1}{2}\\
\mathrm{id} & \frac{1}{2}
\end{cases}.
\]
Then $X\cdot Y_{1}\cdot\ldots\cdot Y_{r-1}$ gives a uniform distribution
on $S_{r}$.
\end{rem}

\section{On Pairs Satisfying $\left({\cal P}1\right)$ and $\left({\cal P}2\right)$
and Further Applications\label{sec:applications}}

In this Section we say a few words about pairs $\left(\Gamma,\pi\right)$
of a group and a representation satisfying properties $\left({\cal P}1\right)$
and/or $\left({\cal P}2\right)$, and elaborate on the combinatorial
applications of Theorem \ref{thm:gamma-pi-one-sided-rmnjn-covering},
alongside the existence of one-sided Ramanujan $r$-coverings as stated
in Theorem \ref{thm:Every-graph-has-one-sided-d-ram-cover}. We begin
with $\left({\cal P}2\right)$, where a complete classification is
known.

\subsection{Complex Reflection Groups}

Recall that the pair $\left(\Gamma,\pi\right)$ satisfies $\left({\cal P}2\right)$
if $\Gamma$ is finite and $\pi\left(\Gamma\right)$ is a complex
reflection group, namely generated by pseudo-reflections: elements
$A\in\mathrm{GL}_{d}\left(\mathbb{C}\right)$ of finite order with
$\mathrm{rank}\left(A-I_{d}\right)=1$. If $\pi$ is not faithful
(not injective), it factors through the faithful $\overline{\pi}\colon\nicefrac{\Gamma}{\ker\pi}\to\mathrm{GL}_{d}\left(\mathbb{C}\right)$,
and $\left(\Gamma,\pi\right)$ satisfies $\left({\cal P}2\right)$
if and only if $\left(\nicefrac{\Gamma}{\ker\pi},\overline{\pi}\right)$
does. In addition, if $\pi$ is reducible and $\left(\Gamma,\pi\right)$
satisfies $\left({\cal P}2\right)$, then for every non-trivial irreducible
component $\pi'$ of $\pi$, the pair $(\Gamma,\pi')$ satisfies $\left({\cal P}2\right)$.

Hence, the classification of pairs satisfying $\left({\cal P}2\right)$
boils down to classifying finite, irreducible complex reflection groups:
finite-order matrix groups inside $\mathrm{GL}_{d}\left(\mathbb{C}\right)$
which are generated by pseudo-reflections and have no non-zero invariant
proper subspaces of $\mathbb{C}^{d}$. This classification was established
in 1954 by Shephard and Todd:
\begin{thm}
\label{thm:shepahrd-todd}\cite{shephard1954finite} Any finite irreducible
complex reflection group $W$ is one the following:

\begin{enumerate}
\item $W\le\mathrm{GL}_{d}\left(\mathbb{C}\right)$ is isomorphic to $S_{d+1}$
$\left(d\ge2\right)$, via the standard representation of $S_{d+1}$
(see the paragraph preceding Claim \ref{claim:equivalence-of-coverings-and-std-coverings}).
\item $W=G\left(m,k,d\right)$ with $m,d\in\mathbb{Z}_{\ge2}$, $k\in\mathbb{Z}_{\ge1}$
and $k|m$. This is a generalization of signed permutations groups:
the group $G\left(m,k,d\right)\le\mathrm{GL}_{d}\left(\mathbb{C}\right)$
consists of monomial matrices (matrices with exactly one non-zero
entry in every row and every column), the non-zero entries are $m$-th
roots of unity (not necessarily primitive), and their product is a
$\frac{m}{k}$-th root of unity. This a group of order $\frac{d!\cdot m^{d}}{k}$.
For example, 
\[
\left(\begin{array}{ccc}
0 & \zeta & 0\\
0 & 0 & \zeta^{-1}\\
\zeta^{4} & 0 & 0
\end{array}\right)
\]
where $\zeta=e^{\frac{2\pi i}{6}}$, is an element of $G\left(6,2,3\right)$.
\item $W=\nicefrac{\mathbb{Z}}{m\mathbb{Z}}\le\mathrm{GL}_{1}\left(\mathbb{C}\right)$
with $m\in\mathbb{Z}_{\ge2}$, sometimes denoted $G\left(m,1,1\right)$,
is the cyclic group of order $m$ whose elements are $m$-th roots
of unity.
\item $W$ is one of $34$ exceptional finite irreducible complex reflection
groups of different dimensions $d$, $2\le d\le8$.
\end{enumerate}
\end{thm}

We remark that a finite complex reflection group which is conjugate
to a subgroup of $\mathrm{GL}_{d}\left(\mathbb{R}\right)$ (matrices
with real entries) is, by definition, a finite Coxeter group. All
groups listed in the theorem are finite complex reflection groups,
and all irreducible except for $G\left(2,2,2\right)$.
\begin{thm}[{Steinberg, \cite[Theorem~4.6]{geck2003reflection}}]
\label{thm:P2 =00003D=00003D> P1} If $\left(\Gamma,\pi\right)$
satisfies $\left({\cal P}2\right)$ and $\pi$ is irreducible, then
$\left(\Gamma,\pi\right)$ satisfies $\left({\cal P}1\right)$ as
well.\\
Namely, if $\Gamma$ is a finite group, $\pi\colon\Gamma\to\mathrm{GL}_{d}\left(\mathbb{C}\right)$
an irreducible representation and $\pi\left(\Gamma\right)$ is a complex
reflection group, then the exterior powers $\bigwedge^{m}\pi$, $0\le m\le d$,
are irreducible and non-isomorphic.
\end{thm}

It is evident that if $\pi$ is reducible, the pair $\left(\Gamma,\pi\right)$
does not satisfy $\left({\cal P}1\right)$. Thus,
\begin{cor}
\label{cor:pairs with both P1 and P2}The pairs $\left(\Gamma,\pi\right)$
satisfying both $\left({\cal P}1\right)$ and $\left({\cal P}2\right)$
are precisely the irreducible finite complex reflection groups\footnote{To be precise, this is true for \textit{\emph{faithful}}\textit{ representations.
If $\pi$ factors through $\overline{\pi}\colon\nicefrac{\Gamma}{\ker\pi}\to\mathrm{GL}_{d}\left(\mathbb{C}\right)$,
then $\left(\Gamma,\pi\right)$ satisfies $\left({\cal P}1\right)$
and $\left({\cal P}2\right)$ if and only if $\left(\nicefrac{\Gamma}{\ker\pi},\overline{\pi}\right)$
does.}}. 
\end{cor}

We conclude:
\begin{cor}
\label{cor:P2 alone}Let $\left(\Gamma,\pi\right)$ satisfy $\left({\cal P}2\right)$
such that $\pi$ has no trivial component in its decomposition to
irreducibles, and let $G$ be a finite graph with no loops. Then,

\begin{itemize}
\item If $d_{1},\ldots,d_{r}$ are the dimensions of the irreducible components
of $\pi$, then 
\[
\mathbb{E}_{\gamma\in{\cal C}_{\Gamma,G}}\left[\phi_{\gamma,\pi}\right]=\M_{d_{1},G}\cdot\ldots\cdot\M_{d_{r},G}.
\]
\item $\mathbb{E}_{\gamma\in{\cal C}_{\Gamma,G}}\left[\phi_{\gamma,\pi}\right]$
is real rooted and there is some labeling with smaller largest root.
\item There is a one-sided Ramanujan $\left(\Gamma,\pi\right)$-covering
of $G$.
\end{itemize}
\end{cor}

\begin{proof}
Assume that $\pi\cong\pi_{1}\oplus\ldots\oplus\pi_{r}$ where $\pi_{i}$
is a non-trivial irreducible representation of dimension $d_{i}$.
Clearly, every pseudo-reflection in $\pi(\Gamma)$ is the identity
element in all but one of the $\pi_{i}(\Gamma)$, and the pseudo-reflections
associated with $\pi_{j}(\Gamma)$ generate $\pi_{j}(\Gamma)$. For
$i=1,\ldots,r$, denote $\Gamma_{i}=\nicefrac{\Gamma}{\ker\pi_{i}}$
and by $\overline{\pi_{i}}\colon\Gamma_{i}\to\mathrm{GL}_{d_{i}}\left(\mathbb{C}\right)$
the associated representation. Then $\Gamma=\Gamma_{1}\times\ldots\times\Gamma_{r}$
and $\pi=(\overline{\pi_{1}},1,\ldots,1)+\ldots+(1,\ldots,1,\overline{\pi_{r}})$.
Constructing a $\left(\Gamma,\pi\right)$-covering of a graph is equivalent
to constructing $r$ independent coverings, one for $\left(\Gamma_{i},\overline{\pi_{i}}\right)$
for each $i$. The corollary now follows from Corollary \ref{cor:pairs with both P1 and P2}
and Theorems \ref{thm:P1} and \ref{thm:P2}.
\end{proof}

\subsection{Pairs Satisfying $\left({\cal P}1\right)$}

The list in Theorem \ref{thm:shepahrd-todd} does not exhaust all
pairs $\left(\Gamma,\pi\right)$ (with $\pi$ faithful) satisfying
$\left({\cal P}1\right)$. Even when restricting to finite groups,
there are pairs satisfying $\left({\cal P}1\right)$ but not $\left({\cal P}2\right)$.
A handful of such examples arises from the observation that $\left({\cal P}1\right)$
is preserved by passing to bigger groups:
\begin{claim}
\label{claim: P1 preserved by supergroups}Let $\Gamma$ be a group,
$\pi\colon\Gamma\to\mathrm{GL}_{d}\left(\mathbb{C}\right)$ a representation
and $H\le\Gamma$ a subgroup. If $\left(H,\pi|_{H}\right)$ satisfies
$\left({\cal P}1\right)$ then so does $\left(\Gamma,\pi\right)$.
\end{claim}

\begin{proof}
It is clear that if $\bigwedge^{m}\pi$ cannot have an invariant proper
subspace if $\left(\bigwedge^{m}\pi\right)|_{H}$ has none. An isomorphism
of $\bigwedge^{m}\pi$ and $\bigwedge^{d-m}\pi$ induces an isomorphism
on the same representation restricted to $H$.
\end{proof}
For example, we can increase $\mathrm{std}\left(S_{r}\right)$ by
adding some scalar matrix of finite order $m$ as an extra generator,
and obtain a $d$-dimensional faithful representation of $S_{d+1}\times\nicefrac{\mathbb{Z}}{m\mathbb{Z}}$
which satisfies $\left({\cal P}1\right)$. 

There are also pairs with $\Gamma$ finite which do not contain any
complex reflection group. For instance, consider the index-2 subgroup
$\Gamma$ of $G\left(2,1,3\right)$ where we restrict to \emph{even}
permutation $3\times3$ matrices with $\pm1$ signing of every non-zero
entry. The natural $3$-dimensional representation of this group satisfies
$\left({\cal P}1\right)$, but does not contain any complex reflection
group. We are not aware of a full classification of pairs $\left(\Gamma,\pi\right)$
satisfying $\left({\cal P}1\right)$, even when $\Gamma$ is finite.

There are some interesting examples of pairs $\left(\Gamma,\pi\right)$
satisfying $\left({\cal P}1\right)$ where $\Gamma$ is infinite and
compact. For example, the standard representation $\pi$ of the orthogonal
group $O\left(d\right)$ or of the unitary group $U\left(d\right)$,
is such (by, e.g., Claim \ref{claim: P1 preserved by supergroups}
and the fact one can identify $\mathrm{std}\left(S_{d+1}\right)$
as a subgroup of $O\left(d\right)$ or of $U\left(d\right)$).
\begin{cor}
\label{cor:O(d) and U(d)}Let $\Gamma=O\left(d\right)$ or $\Gamma=U\left(d\right)$,
and let $\pi$ be the standard $d$-dimensional representation. Then,
for every finite graph $G$,
\[
\mathbb{E}_{\gamma\in{\cal C}_{\Gamma,G}}\left[\phi_{\gamma,\pi}\right]=\M_{d,G}.
\]
\end{cor}

\subsection{\label{subsec:Applications-of-Theorem about gamma-pi-coverings}Applications
of Theorem \ref{thm:gamma-pi-one-sided-rmnjn-covering}}

In this section we elaborate the combinatorial consequences of Theorem
\ref{thm:gamma-pi-one-sided-rmnjn-covering} stating that if $\left(\Gamma,\pi\right)$
satisfies both $\left({\cal P}1\right)$ and $\left({\cal P}2\right)$,
then there is a one-sided Ramanujan $\left(\Gamma,\pi\right)$-covering
of $G$ whenever $G$ is finite with no loops. Corollary \ref{cor:pairs with both P1 and P2}
tells us exactly what pairs satisfy the conditions of the theorem.
The most interesting consequence, based on the pair $\left(S_{r},std\right)$,
was already stated as Theorem \ref{thm:Every-graph-has-one-sided-d-ram-cover}:
every $G$ as above has a one-sided Ramanujan $r$-covering for every
$r$. 

Another interesting application stems from one-dimensional representations
(item $\left(3\right)$ in Theorem \ref{thm:shepahrd-todd}):
\begin{cor}
\label{cor:1-dim coverings}For every $m\in\mathbb{Z}_{\ge2}$ and
every loopless\footnote{In this special case it is actually possible to prove the result even
for graphs with loops: the proof of Proposition \ref{prop:avg-char-poly-of-random-coverings-is-real-rooted}
does not break.} finite graph $G$, there is a labeling of the oriented edges of $G$
by $m$-th roots of unity (with $\gamma\left(-e\right)=\gamma\left(e\right)^{-1}$,
as usual), such that the resulting spectrum is one-sided Ramanujan.
\end{cor}

Of course, the result for $m$ follows from the result for $m'$ whenever
$1\ne m'|m$. For $m=2$ this is the main result of \cite{MSS13}.
As this corollary deals only with one-dimensional representations,
the original proof of \cite{MSS13} can be relatively easily adapted
to show it. This was noticed also by \cite{liu2020signatures}. 

Recall that all irreducible representations of abelian groups are
one-dimensional. Therefore, given an abelian group $\Gamma$ and a
finite graph $G$, there is a $\Gamma$-labeling of $G$ which yields
a one-sided Ramanujan $\left(\Gamma,\pi\right)$-covering for any
irreducible representation $\pi$ of $\Gamma$. However, this certainly
does not guarantee the existence of a single $\Gamma$-labeling which
is ``Ramanujan'' for all irreducible representations simultaneously.
In fact, such a $\Gamma$-labeling does not exist in general -- see
Remark \ref{remark:abelian groups}.

Still, in the special case where $\Gamma=\nicefrac{\mathbb{Z}}{3\mathbb{Z}}$
is the cyclic group of order $3$, there are only two non-trivial
representations $\pi_{1}$ and $\pi_{2}$, and one is the complex
conjugate of the other. Hence, $\phi_{\gamma,\pi_{2}}=\phi_{\gamma,\pi_{1}}$
for any $\Gamma$-labeling $\gamma$, and so the spectra are identical,
and we get, as noticed by \cite{chandrasekaran2015constructing,liu2020signatures}\footnote{Interestingly, it is also shown in \cite{chandrasekaran2015constructing}
that every graph has a one-sided Ramanujan 4-covering with cyclic
permutations. This does not seem to follow from the results in the
current paper.}:
\begin{cor}
\label{cor:C3-coverings}Every finite graph $G$ has a one-sided Ramanujan
$3$-covering, where the permutation above every edge is cyclic.
\end{cor}

From the third infinite family of complex reflection groups (item
$\left(2\right)$ in Theorem \ref{thm:shepahrd-todd}), we do not
get any significant combinatorial implications. If $\Gamma=G\left(m,k,d\right)$,
Theorem \ref{thm:gamma-pi-one-sided-rmnjn-covering} guarantees that
every graph has a one-sided Ramanujan ``signed $d$-covering'':
a $d$-covering of $G$ where every oriented edge is then labeled
by an $m$-th root of unity, and such that the product of roots in
every fiber of edges is an $\frac{m}{k}$-th root of unity. But Corollary
\ref{cor:1-dim coverings} shows that \emph{every} $d$-covering of
$G$ can be edge-labeled by $m$-th roots of unity so that the resulting
spectrum is one-sided Ramanujan. If $k<m$, we can label by $\frac{m}{k}$-th
roots, so applying Theorem \ref{thm:gamma-pi-one-sided-rmnjn-covering}
on $\Gamma$ yields nothing new. If $k=m$, the statement of the theorem
cannot be (easily) derived from former results: we get that $G$ has
a $d$-covering with edges labeled by $m$-th roots of unity, so that
the product of the labels on every fiber is $1$, and the resulting
spectrum is one-sided Ramanujan.

\subsection{Permutation Representations\label{subsec:Permutation-Representations}}

Every group action of $\Gamma$ on a finite set $X$ yields a representation
$\pi$ of dimension $\left|X\right|$. In this case, $\pi$ can be
taken to map $\Gamma$ into permutation matrices, hence $\left(\Gamma,\pi\right)$-coverings
of a graph $G$ correspond to topological $\left|X\right|$-coverings
of $G$ (with permutations restricted to the set $\pi\left(\Gamma\right)$).
Such representations are called permutation representations. For instance,
the natural action of $S_{r}$ on $\left\{ 1,\ldots,r\right\} $ yields
the set of all $r$-coverings from Theorem \ref{thm:Every-graph-has-one-sided-d-ram-cover}.
The action of $\nicefrac{\mathbb{Z}}{3\mathbb{Z}}$ by cyclic shifts
on a set of size 3 yields the regular representation of this group
and the coverings in Corollary \ref{cor:C3-coverings}. In general,
the regular representation of a group is always of this kind.

It is interesting to consider the set ${\cal A}$ of all possible
pairs $\left(\Gamma,\pi\right)$ where $\Gamma$ is a finite group
and $\pi$ a permutation representation, so that every graph has a
(one-sided) Ramanujan $\left(\Gamma,\pi\right)$-covering. Of course,
the action of $\Gamma$ on $X$ must be transitive: otherwise, the
coverings are never connected. Observe this set is closed under two
``operations'':
\begin{enumerate}
\item If $\Lambda\le\Gamma$ and $\left(\Lambda,\pi|_{\Lambda}\right)$
is in ${\cal A}$, then so is $\left(\Gamma,\pi\right)$.
\item The set ${\cal A}$ is closed under towers of coverings: a Ramanujan
covering of a Ramanujan covering is a Ramanujan covering of the original
graph. In algebraic terms this corresponds to wreath products of groups.
Namely, if $\left(\Gamma,\pi\right)$ and $\left(\Lambda,\rho\right)$
are both in ${\cal A}$ with respect to actions on the sets $X$ and
$Y$, respectively, then so is the pair $\left(\Gamma\mathrm{wr}_{X}\Lambda,\psi\right)$,
where 
\[
\Gamma\mathrm{wr}_{Y}\Lambda=\left(\prod_{y\in Y}\Gamma_{y}\right)\rtimes\Lambda
\]
 is the restricted wreath product ($\Gamma_{y}$ is a copy of $\Gamma$
for every $y\in Y$, and $\Lambda$ acts by permuting the copies according
to its action on $Y$), and $\psi$ is based on the action of $\Gamma\mathrm{wr}_{Y}\Lambda$
on the set $X\times Y$ by 
\[
\left(\left\{ g_{y}\right\} ,\ell\right).\left(x,y\right)=\left(g_{y}.x,\ell.y\right).
\]
\end{enumerate}
In this language, for example, a tower of 2-coverings, as considered
by \cite{BL06} and \cite{MSS13}, corresponds to a pair $\left(\Gamma,\pi\right)$
with $\Gamma$ a nested wreath product of $\nicefrac{\mathbb{Z}}{2\mathbb{Z}}$.
See also \cite[Chapter 5]{makelov2015expansion} and the references
therein.

\section{Open Questions\label{sec:Open-Questions}}

We finish with some open questions arising naturally from the discussion
in this paper.

\begin{ques}\textbf{Irreducible representations and one-sided Ramanujan
coverings:} Which pairs $\left(\Gamma,\pi\right)$ of a finite group
and an irreducible representation guarantee the existence of one-sided
Ramanujan $\left(\Gamma,\pi\right)$-coverings for every finite graph?
Can the statement of Theorem \ref{thm:gamma-pi-one-sided-rmnjn-covering}
be extended to more pairs? Does $\left({\cal P}1\right)$ suffice?
In fact, we are not aware of a single example of a pair $\left(\Gamma,\pi\right)$
with $\pi$ non-trivial and irreducible and a finite graph $G$ so
that there is no (one-sided) Ramanujan $\left(\Gamma,\pi\right)$-covering
of $G$. See also Remark \ref{remark:psl2}.

\end{ques}

\begin{ques}\textbf{Full Ramanujan coverings:} The previous question
can be asked for full (two-sided) Ramanujan coverings as well. The
difference is that in this case nothing is known for general graphs.
The case $\left(\nicefrac{\mathbb{Z}}{2\mathbb{Z}},\pi\right)$ with
$\pi$ the non-trivial one-dimensional representation is the Bilu-Linial
Conjecture \cite{BL06}. Proving it would yield the existence of infinitely
many $k$-regular non-bipartite Ramanujan graphs for every $k\ge3$.

\end{ques}

\begin{ques}\textbf{Even ``Fuller'' Ramanujan coverings:} There
is another, stronger sense, of Ramanujan graphs: graphs where the
non-trivial spectrum is contained in the spectrum of the universal
covering tree. The spectrum of the $k$-regular tree is precisely
$\left[-2\sqrt{k-1},2\sqrt{k-1}\right]$, so this coincides with the
standard definition of Ramanujan. But in other families of graphs,
the spectrum of the tree is not necessarily connected, and then the
current definition is stronger. For example, if $k\ge\ell$, the spectrum
of the $\left(k,\ell\right)$-biregular tree is 
\[
\left[-\sqrt{k-1}-\sqrt{\ell-1},-\sqrt{k-1}+\sqrt{\ell-1}\right]\cup\left\{ 0\right\} \cup\left[\sqrt{k-1}-\sqrt{\ell-1},\sqrt{k-1}+\sqrt{\ell-1}\right].
\]
Does every graph have a Ramanujan $r$-covering (or $2$-covering)
in this stricter sense?

\end{ques}

\begin{ques}\textbf{Regular representations and Cayley graphs:} Let
$\Gamma$ be finite and ${\cal R}$ its regular representation. Such
pairs are especially interesting as $\left(\Gamma,{\cal R}\right)$-coverings
of graphs generalize the notion of Cayley graphs (these are $\left(\Gamma,{\cal R}\right)$-coverings
of bouquets). For example, for certain families of finite groups,
mostly simple groups of Lie type, it is known that random Cayley graphs
are expanding uniformly (e.g.~\cite{bourgain2008uniform,breuillard2015expansion}).
Can this be extended to random $\left(\Gamma,{\cal R}\right)$-coverings
of graphs, at least, say, when $G$ is a good expander itself?

\end{ques}

\begin{ques}\textbf{\label{ques:The-d-matching-polynomial}The $\mathbf{d}$-matching
polynomial:} This paper shows that $\M_{d,G}$, the $d$-matching
polynomial of the graph $G$ share quite a few properties with the
classical matching polynomial, $\M_{1,G}$. But $\M_{1,G}$ has many
more interesting properties (a good reference is \cite{godsil1993algebraic}).
What parts of this theory can be generalized to $\M_{d,G}$? In particular,
it would be desirable to find a more direct proof to the real-rootedness
of $\M_{d,G}$. Such a proof may allow to extend this real-rootedness
result to graph with loops.

\end{ques}

\begin{ques}\textbf{\label{ques:Dealing-with-loops}Loops:} Some
results in this paper hold for graphs with loops (e.g.~Theorem \ref{thm:P1}).
We conjecture that, in fact, all the results hold for graphs with
loops. In particular, we conjecture that any finite graph $G$ with
loops should have a one-sided Ramanujan $r$-covering (Theorem \ref{thm:Every-graph-has-one-sided-d-ram-cover}),
that $\M_{d,G}$ is real-rooted for every $d$ (Theorem \ref{thm:M_d,G is Ramanujan})
and that if $\left(\Gamma,\pi\right)$ satisfies $\left({\cal P}1\right)$
and $\left({\cal P}2\right)$, then $G$ has a one-sided Ramanujan
$\left(\Gamma,\pi\right)$-covering (Theorem \ref{thm:gamma-pi-one-sided-rmnjn-covering}).
(And see Question \ref{ques:Another-interlacing-family}.)

If true, this would yield, for example, that if $A$ is a uniformly
random permutation matrix, or Haar-random orthogonal or unitary matrix
in $U\left(d\right)$, then $\mathbb{E}\left[\phi\left(A+A^{*}\right)\right]$
is real-rooted.

\end{ques}

\begin{ques}\textbf{\label{ques:Another-interlacing-family}Another
interlacing family of characteristic polynomials:} The one argument
in this paper that breaks for loops is in the proof of Proposition
\ref{prop:avg-char-poly-of-random-coverings-is-real-rooted}. The
problem is that if $e$ is a loop, then $\pi\left(\gamma\left(e\right)\right)$
and $\pi\left(\gamma\left(-e\right)\right)$ lie in the same $d\times d$
block of $A_{\gamma,\pi}$. One way to extend the arguments for loops
is to prove the following parallel of Theorem \ref{thm:MSS15}, which
we believe should hold:

For a matrix $A$ denote $A^{\herm}\overset{\mathrm{def}}{=}A+A^{*}$.
If ${\cal W}=\left\{ W_{i,j}\right\} _{1\le i\le m,1\le j\le\ell\left(i\right)}$
is defined as in Theorem \ref{thm:MSS15}, then
\[
\mathbb{E}_{{\cal W}}\left[\phi\left(\left[W_{1,1}\ldots W_{1,\ell\left(1\right)}A_{1}+\ldots+W_{m,1}\ldots W_{m,\ell\left(m\right)}A_{m}\right]^{\herm}\right)\right]
\]
is real-rooted for every $A_{1},\ldots,A_{m}\in\mathrm{M}_{d}\left(\mathbb{C}\right)$.

If true, this statement generalizes the fact that the characteristic
polynomials $\phi\left(A+A^{*}\right)$ and $\phi\left(BA+\left(BA\right)^{*}\right)$
interlace whenever $A,B\in\mathrm{GL}_{d}\left(\mathbb{C}\right)$
with $B$ a pseudo-reflection.

\end{ques}

\section*{Acknowledgments}

We would like to thank Miklós Abért, Péter Csikvári, Nati Linial and
Ori Parzanchevski for valuable discussions regarding some of the themes
of this paper. We also thank Daniel Spielman for sharing with us an
early version of \cite{MSS15}. Finally, we thank Jesse Geneson for
pointing out an inaccuracy in the statement of Corollary \ref{cor:P2 alone}
in the previous (and published) version of this paper.

\bibliographystyle{amsalpha}
\bibliography{ramanujan-coverings}

\noindent Chris Hall,\\
Department of Mathematics,\\
University of Wyoming \\
Laramie, WY 82071 USA \\
chall14@uwyo.edu\\

\noindent Doron Puder, \\
School of Mathematics,\\
Institute for Advanced Study,\\
Einstein Drive, Princeton, NJ 08540 USA\\
doronpuder@gmail.com\\

\noindent William F. Sawin, \\
Department of Mathematics, \\
Princeton University\\
Fine Hall, Washington Road \\
Princeton NJ 08544-1000 USA\\
wsawin@math.princeton.edu
\end{document}